\documentclass[11pt]{amsart}
\usepackage{amsmath}
\usepackage{amscd}
\usepackage{amssymb}
\usepackage{amsthm}
\usepackage{dsfont} 
\usepackage{appendix} 
\usepackage{verbatim}
\usepackage[normalem]{ulem} 
\usepackage[all]{xy} 
\usepackage{mathtools}
\usepackage{stmaryrd}

\begin{document}


\theoremstyle{plain}
\newtheorem{theorem}{Theorem}[section]
\newtheorem{introtheorem}{Theorem}

\theoremstyle{plain}
\newtheorem{proposition}[theorem]{Proposition}
\newtheorem{prop}[theorem]{Proposition}

\theoremstyle{plain}
\newtheorem{corollary}[theorem]{Corollary}

\theoremstyle{plain}
\newtheorem{lemma}[theorem]{Lemma}

\theoremstyle{plain}
\newtheorem{expectation}[theorem]{Expectation}

\theoremstyle{remark}
\newtheorem{remind}[theorem]{Reminder}

\theoremstyle{definition}
\newtheorem{condition}[theorem]{Condition}

\theoremstyle{definition}
\newtheorem{construction}[theorem]{Construction}

\theoremstyle{definition}
\newtheorem{definition}[theorem]{Definition}

\theoremstyle{definition}
\newtheorem{question}[theorem]{Question}

\theoremstyle{definition}
\newtheorem{example}[theorem]{Example}

\theoremstyle{definition}
\newtheorem{notation}[theorem]{Notation}

\theoremstyle{definition}
\newtheorem{convention}[theorem]{Convention}

\theoremstyle{definition}
\newtheorem{conj}[theorem]{Conjecture}

\theoremstyle{definition}
\newtheorem{assumption}[theorem]{Assumption}
\newtheorem{isoassumption}[theorem]{Isomorphism Assumption}

\newtheorem{finassumption}[theorem]{Finiteness Assumption}
\newtheorem{indhypothesis}[theorem]{Inductive Hypothesis}

\theoremstyle{definition}
\newtheorem{remark}[theorem]{Remark}

\numberwithin{equation}{section}





\newcommand{\mathscr}[1]{\mathcal{#1}}
\newcommand{\TODO}{{\color{red} TODO}}
\newcommand{\CHANGE}{{\color{red} CHANGE}}
\newcommand{\annette}[1]{{\color{red}Annette: #1 }}
\newcommand{\simon}[1]{{\color{green}Simon: #1 }}
\newcommand{\giuseppe}[1]{{\color{blue}Giuseppe: #1 }}

\newcommand{\tensor}{\otimes}
\newcommand{\Fil}{\tn{Fil}}
\newcommand{\Gh}{\mathcal{G}}
\newcommand{\Fh}{\shfF}
\newcommand{\Oh}{\mathcal{O}}
\newcommand{\GrVec}{\textnormal{GrVec}}
\newcommand{\SO}{\textnormal{SO}}
\newcommand{\REP}{\textnormal{REP}}
\newcommand{\MOD}{\textnormal{MOD}}
\newcommand{\crys}{\textnormal{crys}}
\newcommand{\GSp}{\textnormal{GSp}}
\newcommand{\Gal}{\textnormal{Gal}}
\newcommand{\Var}{\textnormal{Var}}
\newcommand{\BN}{\textnormal{BN}}
\newcommand{\modulo}{\textnormal{mod}}
\newcommand{\op}{\textnormal{op}}
\newcommand{\CH}{\mathrm{CH}}
\newcommand{\im}{\mathrm{im}}
\newcommand{\cl}{\mathrm{cl}}
\newcommand{\sing}{\mathrm{sing}}
\newcommand{\prim}{\mathrm{prim}}
\newcommand{\Clo}{\mathrm{Clo}}
\newcommand{\dR}{\mathrm{dR}}
\newcommand{\CHM}{\mathrm{CHM}}
\newcommand{\GL}{\mathrm{GL}}
\newcommand{\NUM}{\mathrm{NUM}}
\newcommand{\HK}{\mathrm{HK}}
\newcommand{\Frac}{\mathrm{Frac}}
\newcommand{\inv}{\mathrm{inv}}
\newcommand{\ab}{\mathrm{ab}}
\newcommand{\num}{\mathrm{num}}

\newcommand{\id}{\mathrm{id}}
\newcommand{\isocan}{\xrightarrow{\hspace{1.85pt}\sim \hspace{1.85pt}}}
\newcommand{\isom}{\cong}
\newcommand{\red}{\mathrm{red}}
\newcommand{\Betti}{R_B}
\newcommand{\ladic}{R_\ell}
\newcommand{\Hodge}{R_H}
\newcommand{\MHM}{\mathrm{MHM}}
\newcommand{\Hh}{\mathcal{H}}
\newcommand{\A}{\mathbb{A}}
\newcommand{\gm}{\mathrm{gm}}
\newcommand{\mot}{\mathrm{mot}}
\newcommand{\qfh}{\mathrm{qfh}}
\newcommand{\DM}{\tn{\tbf{DM}}}
\newcommand{\DA}{\tn{\tbf{DA}}}
\newcommand{\DMeff}{\tn{\tbf{DM}}^{\eff}}
\newcommand{\DAeff}{\tn{\tbf{DA}}^{\eff}}
\newcommand{\DAgm}{\tn{\tbf{DA}}_{\gm}}
\newcommand{\Bei}{\mathcyr{B}}
\newcommand{\kd}{\tn{kd}}
\newcommand{\Sm}{\tn{\tbf{Sm}}}
\newcommand{\SmCor}{\tn{\tbf{SmCor}}}
\newcommand{\cor}{\tn{\tbf{cor}}}
\newcommand{\Mor}{\textnormal{Mor}}
\newcommand{\Hom}{\textnormal{Hom}}
\newcommand{\End}{\textnormal{End}}
\newcommand{\sss}{\textnormal{ss}}
\newcommand{\Sh}{\tn{\tbf{Sh}}}
\newcommand{\et}{\tn{\'{e}t}}
\newcommand{\an}{\tn{an}}
\newcommand{\D}{\tn{D}}
\newcommand{\eff}{\tn{eff}}
\newcommand{\DMgm}{\DM_{\tn{gm}}}
\newcommand{\vp}{\varphi}
\newcommand{\Sym}{\textnormal{Sym}}
\newcommand{\OSym}{\textnormal{coSym}}
\newcommand{\Mof}{M_1}
\newcommand{\CGS}{\tn{\tbf{cGrp}}}
\newcommand{\tr}{\textnormal{tr}}
\newcommand{\one}{\mathds{1}}
\newcommand{\Spec}{\textnormal{Spec}}
\newcommand{\Sch}{\tn{\tbf{Sch}}}
\newcommand{\Nor}{\tn{\tbf{Nor}}}
\newcommand{\Lie}{\tn{Lie}}
\newcommand{\Ker}{\tn{Ker}}
\newcommand{\Frob}{\tn{Frob}}
\newcommand{\Ver}{\tn{Ver}}
\newcommand{\N}{\mathbb{N}}
\newcommand{\Z}{\mathbb{Z}}
\newcommand{\Q}{\mathbb{Q}}
\newcommand{\Ql}{\mathbb{Q}_{\ell}}
\newcommand{\R}{\mathbb{R}}
\newcommand{\C}{\mathbb{C}}
\newcommand{\G}{\mathbb{G}}

\newcommand{\xra}{\xrightarrow}
\newcommand{\xla}{\xleftarrow}
\newcommand{\sxra}[1]{\xra{#1}}
\newcommand{\sxla}[1]{\xla{#1}}

\newcommand{\sra}[1]{\stackrel{#1}{\ra}}
\newcommand{\sla}[1]{\stackrel{#1}{\la}}
\newcommand{\slra}[1]{\stackrel{#1}{\lra}}
\newcommand{\sllra}[1]{\stackrel{#1}{\llra}}
\newcommand{\slla}[1]{\stackrel{#1}{\lla}}

\newcommand{\ira}{\stackrel{\simeq}{\ra}}
\newcommand{\ila}{\stackrel{\simeq}{\la}}
\newcommand{\ilra}{\stackrel{\simeq}{\lra}}
\newcommand{\illa}{\stackrel{\simeq}{\lla}}

\newcommand{\ul}[1]{\underline{#1}}
\newcommand{\tn}[1]{\textnormal{#1}}
\newcommand{\tbf}[1]{\textbf{#1}}

\newcommand{\Spt}{\mathop{\mathbf{Spt}}\nolimits}
\newcommand{\ra}{\rightarrow}
\newcommand{\old}{\mathrm{old}}

\setcounter{tocdepth}{1}

\title[Standard conjectures for abelian fourfolds]{Standard conjectures for abelian fourfolds}

\author{Giuseppe Ancona}
\address{Institut de Recherche Math\'ematique Avanc\'ee, Universit\'e de Strasbourg}
\email{ancona@math.unistra.fr}

\begin{abstract}
Let $A$ be an abelian fourfold in characteristic $p$. We prove the standard conjecture of Hodge type for $A$, namely that the intersection product
\[\mathcal{Z}^2_{\num}(A)_{\Q}\times \mathcal{Z}_{\num}^2(A)_{\Q} \longrightarrow \Q\]
is of signature $(\rho_2 - \rho_1 +1; \rho_1 - 1)$, with
$\rho_n=\dim \mathcal{Z}_{\num}^n(A)_{\Q}.$  (Equivalently, it is positive definite when restricted to primitive classes for any choice of the polarization.) The approach consists in reformulating   this question into a $p$-adic problem and then using $p$-adic Hodge theory to solve it.

By combining this result with a theorem of Clozel we deduce that numerical equivalence on $A$ coincides with $\ell$-adic homological equivalence on $A$ for infinitely many prime numbers $\ell$. Hence, what is missing among the standard conjectures for abelian fourfolds  is $\ell$-independency of $\ell$-adic homological equivalence.
\end{abstract}

\maketitle
\begin{center}
\today
\end{center}
\tableofcontents
\section{Introduction}
In this paper we prove that the standard conjecture of Hodge type holds for abelian fourfolds. This is the first  unconditional\footnote{Milne showed that the   Hodge conjecture for complex abelian varieties  implies the standard conjecture of Hodge type for abelian varieties over any field \cite{pola}.}  result on the conjecture since its formulation. Before giving the precise statement and a sketch of the proof, we briefly recall the history of the problem.

In this introduction $X$ will be a smooth, projective and geometrically connected variety over a base field $k$ of characteristic $p\geq 0$. We denote by  $\mathcal{Z}_{\num}^n(X)_{\Q}$ the space of algebraic cycles of codimension $n$ with $\Q$-coefficients modulo numerical equivalence. This is a finite dimensional {$\Q$-vector} space and its dimension will be denoted by
\begin{equation} \label{eq:intro}
\rho_n=\dim \mathcal{Z}_{\num}^n(X)_{\Q}. \tag{$*$}
\end{equation}

\subsection*{Brief historical panorama}
The history of the problem starts with the so called Hodge index  theorem.
\begin{theorem}\label{introsegre}
Suppose that $X$ has dimension two. Then the intersection product  
\[\mathcal{Z}^1_{\num}(X)_{\Q}\times \mathcal{Z}_{\num}^1(X)_{\Q} \longrightarrow \Q\]
is of signature $(s_+ ; s_-)=(1; \rho_1 - 1)$.
\end{theorem}
When the characteristic $p$ is zero, the above theorem was proved by Hodge by relating  the intersection product to the cup product in singular cohomology through the cycle class map \cite{Hodge}. An algebraic proof, valid in any characteristic,  was found by Segre \cite{Segre} and Bronowski \cite{Brono}.

Elaborating on an argument of Mattuck--Tate \cite {Mattuck}, Grothendieck  realized that the Lang--Weil estimate for the number of rational points on a smooth and projective curve $C$ over a finite field follows from the Hodge index theorem applied to the surface $X=C\times C$ \cite{Grothseg}. He then proposed a program to show the Weil conjectures for varieties of higher dimension based on a conjectural generalisation of the Hodge index theorem,  known as the standard conjecture of Hodge type. Together with the three other standard conjectures (and the resolution of singularities), it was considered by Grothendieck as the most urgent task in algebraic geometry \cite{GrothSC}. 

This conjecture is connected with other arithmetic contexts such as the conjectural description of the rational points on a Shimura variety over a finite field \cite{LRSC} and the weight-monodromy conjecture \cite{mSaito}. 

The four standard conjectures imply that the category of numerical motives   is semisimple and polarizable just as the category of Hodge structures of smooth projective complex varieties \cite{Rivano}. Surprisingly enough, Jannsen proved semisimplicity   unconditionally and independently of the standard conjectures \cite{Jann}. Polarizability is still open and it is intimately related to the standard conjecture of Hodge type.

 As a conclusion of this historical panorama, let us mention that Gillet and Soul\'e have proposed an arithmetic version of the standard conjectures (i.e. over a ring of integers) \cite{soule}. Some results on the arithmetic standard conjecture of Hodge type can be found in \cite{Kunz1,Kunz2,Kunz3,Kresch}.
  
  \subsection*{Main results}
  Let us now recall the formulation of the standard conjecture of Hodge type.
  \begin{conj}
  Let   $d$ be the dimension of   a smooth, projective and geometrically connected variety $X$ and fix a hyperplane section $L$ of $X$.  For $n \leq d/2$ define the space of primitive cycles $\mathcal{Z}^{n,\prim}_{\num}(X)_{\Q}$ as
\[\mathcal{Z}^{n,\prim}_{\num}(X)_{\Q} = \{\alpha \in \mathcal{Z}^{n}_{\num}(X)_{\Q} , \hspace{0.2cm} \alpha \cdot L^ {{d-2n+1}}=0 \hspace{0.2cm}  \textrm{in}  \hspace{0.2cm}  \mathcal{Z}^{d-n+1}_{\num}(X)_{\Q}\}\]
and define the pairing  
\[\langle \cdot, \cdot \rangle_{n} : \mathcal{Z}^{n,\prim}_{\num}(X)_{\Q}  \times  \mathcal{Z}^{n,\prim}_{\num}(X)_{\Q}  \longrightarrow \Q\]
 via the intersection product
 \[\alpha, \beta \mapsto  (-1)^n \alpha \cdot \beta \cdot L^{d-2n}.\]
The   standard conjecture of Hodge type predicts that this pairing is positive definite.
\end{conj}

The evidences for Grothendieck were the case $n=1$ (which can be shown by reducing it to Theorem \ref{introsegre}) and the case     $p=0$. Indeed,  in characteristic zero, one can use the cycle class map to relate the quadratic form $\langle \cdot, \cdot \rangle_{n}$ to the quadratic form given by the cup product on singular cohomology. Then  this kind of positivity statements can be deduced from positivity statements in cohomology, such as the Hodge--Riemann relations in Hodge theory.

\

 The following is our main theorem. 

\begin{theorem}\label{thmintro}
The standard conjecture of Hodge type holds for abelian fourfolds in positive characteristic.\end{theorem}
It turns out that this statement is equivalent to the following, which is maybe a more direct formulation (see also  Proposition \ref{anypolarization}).
\begin{theorem} 
Let $X$ be an abelian fourfold. Then the intersection product 
\[\mathcal{Z}^2_{\num}(X)_{\Q}\times \mathcal{Z}_{\num}^2(X)_{\Q} \longrightarrow \Q\]
is of signature  $ (s_+;s_-) = ( \rho_2 - \rho_1+1;  \rho_1-1)$, with $\rho_n$ as in $($\ref{eq:intro}$)$.
\end{theorem} 
This formulation should   be reminiscent of the Hodge index theorem.

\begin{remark}\label{remintro}
As we explained above, the standard conjecture of Hodge type is known in characteristic zero. Hence Theorem \ref{thmintro} is new only   for those algebraic classes that cannot be lifted to characteristic zero. We discuss the existence of such classes in  Appendix \ref{geometricexamples}.
\end{remark}

By combining Theorem \ref{thmintro} with  a theorem of Clozel  \cite{Clozel} we deduce the following.
\begin{theorem}\label{thmclozelintro}
Let $X$ be an   abelian fourfold. Then numerical equivalence on $X$ coincides with $\ell$-adic homological equivalence on $X$ for infinitely many prime numbers $\ell$. 
\end{theorem}
The fact that homological and numerical equivalence should always coincide is also one   of the four standard conjectures. The two others (namely K\"unneth and Lefschetz) being known for abelian varieties, Theorems \ref{thmintro} and \ref{thmclozelintro} imply that in order to fully understand the standard conjectures for abelian fourfolds what is missing is $\ell$-independency of $\ell$-adic homological equivalence.

\subsection*{Idea of the proof}
The starting point of the proof\footnote{An idea of the strategy is also presented in the report  \cite{Ober}.} of Theorem \ref{thmintro} is a classical product formula from the theory of quadratic forms: let $q$ be a $\Q$-quadratic form, if we know  $q \otimes \Q_\ell$ for all prime numbers $\ell$ then we have information on the signature of $q$, more precisely we know the difference $s_+-s_-$ modulo $8$. When $q$ is the quadratic form  $\langle \cdot, \cdot \rangle_{n}$ as above, then one can hope to compute $q \otimes \Q_\ell$ through the cycle class map.

In characteristic $p=0$ one computes $q \otimes \R$ by directly embedding it into singular cohomology. Our strategy should be thought as a way of circumventing the impossibility (when $p>0$) of embedding $q \otimes \R$ in some Weil cohomology by   instead  embedding all the other completions of   $q$ in  Weil cohomologies.

\

In order to deduce the full signature from the information modulo $8$ one needs to be sure that the rank of $q$ is small. Hence, one cannot apply this strategy to the whole space of algebraic cycles but rather one first decomposes it into smaller quadratic subspaces and then computes $q \otimes \Q_\ell$  for each of those.

To do so, one first   reduces the question to varieties defined over a finite field (this is a classical specialization argument) where abelian varieties are known to always  admit complex multiplication \cite{Tate}.  Then one uses complex multiplication to decompose the space of algebraic cycles into smaller quadratic subspaces. Finally,  it turns out that  when the abelian variety has dimension four then   the subspaces constructed are of rank two (at least those where the problem is not trivial).

\

Technically speaking, we do not compute $q \otimes \Q_\ell$ directly, but rather construct another $\Q$-quadratic form $\tilde{q}$ and try to compare these two. This quadratic form $\tilde{q}$ is constructed as follows. First, the abelian fourfold $X$ admits a lifting $\tilde{X}$ to characteristic zero on which  complex multiplication still acts (Honda--Tate). The action of  complex multiplication  decomposes the singular cohomology  $H_{\sing}^4(\tilde{X},\Q)$ into subspaces. Those are quadratic subspaces endowed with the cup product. To a given factor $q$ of $\mathcal{Z}^2_{\num}(X)_{\Q}$ one associates the factor $\tilde{q}$  of $H_{\sing}^4(\tilde{X},\Q)$ which is (roughly speaking) the same irreducible representation for the action of  complex multiplication.

\

The comparison  $q \otimes \Q_\ell \cong \tilde{q}\otimes \Q_\ell$ holds for all $\ell \neq p$   by smooth proper base change in $\ell$-adic cohomology.  On the other hand one computes $\tilde{q} \otimes \R$ using the Hodge--Riemann relations; it is negative definite if and only if the Hodge types appearing in the Hodge structure $\tilde{q}$ are odd. 

Now the theory of quadratic forms (in particular the product formula on Hilbert symbols) tells us that the positivity of $q \otimes \R$ is equivalent to the following statement. The quadratic forms $q \otimes \Q_p $ and $\tilde{q}\otimes \Q_p$ are   not isomorphic  precisely   if the Hodge types of $\tilde{q}$ are odd. (Note that this equivalence holds because our quadratic spaces have rank two.)

\

This formulation translates the  problem into a question in $p$-adic Hodge theory. To solve it we use the $p$-adic comparison theorem which gives a canonical isomorphism $(q\otimes \Q_p) \otimes B_\crys \cong (\tilde{q}\otimes \Q_p) \otimes B_\crys$ over a   $\Q_p$-algebra $B_{\crys}$. The strategy then consists in  writing explicitly the matrix of the isomorphism with respect to  well chosen bases of the two $\Q_p$-structures (i.e. computing the $p$-adic periods) and then exploit this explicit isomorphism to determine whether the two quadratic forms are  isomorphic also over $\Q_p$. 

The reason for which  a $p$-adic period is in general more computable than a complex one is that it comes equipped with some extra structures (in particular the action of an absolute Frobenius and the de Rham filtration) which sometimes characterize it\footnote{For example, an element of $B_\crys$ which is invariant under the Frobenius and sits in the zeroth step  of the filtration must be an element of $\Q_p$  \cite[Theorem 5.3.7]{Fontp}.}. Moreover, in our particular situation, these extra structures are particularly simple: the absolute Frobenius must act trivially (because $q\otimes \Q_p$ is spanned by algebraic cycles) and the only non-trivial subspace appearing in the de Rham filtration is a line (as these quadratic spaces have rank two).

Once   the matrix $f \in M_{2 \times 2} (B_\crys)$ of $p$-adic periods is explicitly computed we exploit the relation $(q\otimes B_\crys)=  (\tilde{q}\otimes B_\crys) \circ f$. Even though in this equality we only know $f$, this information is enough to determine whether the two quadratic forms  $q\otimes \Q_p$ and $\tilde{q}\otimes \Q_p$ have the same discriminant and represent the same elements of $\Q_p$. Again because we are in rank two these properties control whether  $q\otimes \Q_p$ and $\tilde{q}\otimes \Q_p$ are isomorphic.

\subsection*{Organisation of the paper}
In Section $3$ we recall the standard conjecture of Hodge type and present some first reduction steps, among which the fact that it is enough to work over a finite field. We state Theorem \ref{thmintro} (Theorem \ref{thmscht} in the text) and deduce from it Theorem \ref{thmclozelintro} (Theorem \ref{thmclozel} in the text). In Section $4$ we recall classical results on the motive of an abelian variety. In Section $5$ we show that Theorem \ref{thmscht} holds true for Lefschetz classes on abelian varieties, i.e. those algebraic classes that are linear combinations of intersections of divisors. In Section $6$ we recall that an abelian variety $X$ over a finite field admits complex multiplication and use this to decompose its   motive. We also recall that $X$ together with its complex multiplication lifts to characteristic zero, hence its motivic decomposition lifts too. In Section $7$ we study the interesting subfactors of this decomposition and we call them exotic. By definition, they are those containing algebraic classes which are not Lefschetz classes. The main result of the section says that if $X$ has dimension four then exotic motives have rank two (in the sense that their realizations are cohomology groups of rank two).

From there a somehow independent text starts, in which abelian varieties are not involved anymore: we study motives of rank two, living in mixed characteristic and having algebraic classes in positive characteristic. The main result (Theorem \ref{mainthm}) predicts the signature of the intersection product on those algebraic classes. In Section $8$ we put all pieces together and explain how Theorem \ref{mainthm} implies Theorem \ref{thmscht}. In Section $9$ we recall classical facts   from the theory of quadratic forms and use them to reduce Theorem \ref{mainthm}  to a $p$-adic question. The tool to attack this $p$-adic question is $p$-adic Hodge theory which we recall in Section $10$. In Section $11$ we describe the extra-structures (Frobenius and filtrations) on the   $p$-adic periods appearing and we uniquely characterize them with respect to these extra-structures. This characterization is essential to be able to determine them. We will compute them in  Section $12$. The proof of Theorem \ref{mainthm} is concluded in Section $13$.

Appendix \ref{geometricexamples} contains exemples of exotic motives and in particular of those having algebraic classes that cannot be lifted to characteristic zero.
\subsection*{Acknowledgements}
I would like to thank  Olivier Benoist,  Rutger Noot and Jean-Pierre Wintenberger for useful discussions; Eva Bayer, Giancarlo Lucchini Arteche and Olivier Wittenberg for explanations on  Galois cohomology; Xavier Caruso and Matthew Morrow for explanations on comparison theorems;   Olivier Brinon and Adriano Marmora for their help with $p$-adic periods; Robert Laterveer and Gianluca Pacienza for drawing my attention to the example in Proposition  \ref{propfermat}; Fr\'ed\'eric D\'eglise for pointing out the reference \cite{NiziFred}; Javier Fres\'an, Thomas Kr\"amer, Marco Maculan and Anastasia Prokudina for useful comments on the text.

Finally, I would like to thank the anonymous referee for many useful suggestions and especially for pointing out an argument which allowed to compute the $p$-adic periods in a direct way instead of using \cite[\S 9]{Colmez}. This made the $p$-adic part of the paper more self-contained and as elementary as possible. 

This research was partly supported by the grant ANR--18--CE40--0017 of Agence National de la Recherche. 
\section{Conventions}
Throughout the paper we will use the following conventions.
\begin{enumerate}
\item  A \textbf{variety} will mean a smooth, projective and geometrically connected scheme over a base  field  $k$.

\item Let  $k$ be a base field and $F$ be a field of coefficients of characteristic zero, we will denote by \[\CHM(k)_{F}\] the category of  \textbf{Chow motives} over $k$ with coefficients in $F$. (The subscript $F$ may sometimes be omitted.) For generalities, we refer to \cite{Andmot} in particular to \cite[Definition 3.3.1.1]{Andmot}  for the notion of Weil cohomology and to \cite[Proposition 4.2.5.1]{Andmot}  for the associated realization functor. 

We will denote by
\[\mathfrak{h} : \Var_k\longrightarrow \CHM(k)_{F}^{\op}\]
the functor associating to each variety its motive.

We will work also with \textbf{homological and numerical motives}. The motive of a variety $X$ in these categories will still be denoted by $\mathfrak{h} (X)$. We will denote by 
\[\NUM(k)_F\]
the category of numerical motives over $k$ with coefficients in $F$. 

Finally, we will need also to work with  \textbf{relative motives} as defined for example in \cite[\S 5.1]{OS2}.  The relative situation that will appear in this paper will always be over the ring of integers $W$ of a $p$-adic field. We will use analogous notation as in the absolute setting, for example Chow motives over $W$ will be denoted by \[\CHM(W).\]

\item By \textbf{classical realizations}\footnote{We tried to distinguish the properties which hold true for any realization and those which are known only for classical realizations.  In any case, the main results of the paper only make use of classical realizations and the reader can safely think only about them.} we will mean Betti realization
\[R_B : \CHM(\C)_{\Q}\longrightarrow \GrVec_\Q,\]
  $\ell$-adic realization (when $\ell$ is invertible in $k$)
\[R_\ell : \CHM(k)_{\Q}\longrightarrow \GrVec_{\Q_\ell}\]
 de Rham realization (when $k$ is of characteristic zero)
\[R_\dR : \CHM(k)_{\Q}\longrightarrow \GrVec_{k}\]
and crystalline realization (when $k$ is perfect and of positive characteristic)
\[R_\crys : \CHM(k)_{\Q}\longrightarrow \GrVec_{\Frac (W(k))}.\]
 For generalities see  \cite[\S 3.4.2 and Proposition 4.2.5.1]{Andmot}.

The realization functor are sometimes used in their enriched form, i.e. $R_B$ provides Hodge structures, $R_\ell$ provides Galois representations, $R_\dR$ provides filtered vector spaces and $R_\crys$ provides absolute Frobenius actions, see \cite[\S 7.1]{Andmot}.

\item\label{Hyodo-Kato} Let $W$ be the ring of integers  of a $p$-adic field $K$ with residue field $k$ (then $\Frac (W(k))$ is the maximal unramified subfield of $K$). We will denote by  \[\MOD\] the category  of admissible filtered $\varphi$-modules \cite{BertOgus,HyodoKato,NiziFred}.

Recall that an element in $\MOD$ is, in particular, a finite dimensional  $\Frac (W(k))$-vector space $V$, together with a decreasing filtration $\Fil^*$ on $V\otimes_{\Frac (W(k))} K$ and a $(\Frac (W(k))/\Q_p)$-semilinear endomorphism $\varphi$ of $V$, usually called absolute Frobenius, verifying some compatibilities between $\Fil^*$ and $\varphi$. 

There is a \textbf{Hyodo-Kato} realization functor \[R_{\HK}: \CHM(W)\rightarrow \MOD\]
such that, after forgetting the absolute Frobenius, it coincides with de Rham realization of the generic fiber $R_{\HK}(\cdot)\otimes K=R_\dR(\cdot_{\vert K})$  and, after forgetting the filtration, it coincides with  crystalline realization of the special fiber $R_{\HK}(\cdot) = R_{\crys}(\cdot_{\vert k})$. 
\item The \textbf{unit object} in one of the above categories of motives (Chow, homological or numerical motives) over a base $S$ is the motive $\mathfrak{h}(S)$ and it will be denoted by
\[\one \coloneqq \mathfrak{h}(S).\]
When $S=\Spec(k)$ we have the identification
\[\End(\one)=\Q.\]
\item An \textbf{algebraic class} of a motive $T$ is a map \[\alpha\in\Hom(\one, T).\]
The realization of an algebraic class gives a map
\[R(\alpha) : F \longrightarrow R(T),\]
where $F$ is the field of coefficients of the realization. 
The map $R(\alpha)$ is characterized by its value at $1\in F$. An element of the cohomology group $R(T)$ of this form
\[R(\alpha)(1) \in R(T)\]
will be called an algebraic class of  $R(T)$. By abuse of language, an algebraic class of the cohomology group $R(T)$ may also be called an algebraic class of the motive $T$.

Finally, we  may sometimes ignore Tate twists and say that $T$ has algebraic classes when $T(n)$ does. For example, for a variety $X$, the cohomology group $H^{2n}(X)$ may have algebraic classes, although, strictly speaking, these classes belong to $H^{2n}(X)(n).$ Hopefully this will not bring confusion as we will work with motives whose realization is concentrated in one single cohomological degree.

\item An algebraic class $\alpha$ of a motive $T$ is \textbf{numerically trivial} if, for all $\beta\in\Hom(T,\one)$, the composition
$\beta \circ \alpha $  is zero. Two algebraic classes are numerically equivalent if their difference is numerically trivial. (This recovers the classical definition \cite[7.1.2]{AK}.)

\item Let $X$ be a variety. The space of algebraic cycles on $X$ with coefficients in $F$ modulo numerical equivalence will be denoted by \[\mathcal{Z}^{n}_{\num}(X)_{F}= \Hom_{\NUM(k)_F}(\one, \mathfrak{h}(X)(n)).\]
Similarly, for a fixed Weil cohomology, $\mathcal{Z}^{n}_{\hom}(X)_{F}$ will denote the space of algebraic cycles on $X$ with coefficients in $F$ modulo homological equivalence.

\end{enumerate}

\section{Standard conjecture of Hodge type}

In this section we recall some classical conjectures (due to Grothendieck) and give some reformulations of those. We state our main result (Theorem \ref{thmscht}) and prove a consequence (Theorem \ref{thmclozel}).

Throughout the section, $k$ is a base field, $H^*$ is a fixed classical Weil cohomology and $R$ is the associated realization functor (see Conventions). We fix a variety $X$  over $k$ of dimension $d$ and  an ample divisor $L$ on $X$. We will write $\mathfrak{h}(X)$ for its homological motive (with respect to $H^*$). The section is written under the Assumption \ref{kun+lef} (which is satisfied in particular by abelian varieties).

\begin{conj}\label{conjkun} (Standard conjecture of K\"unneth type)

\noindent There exists a decomposition 
\[\mathfrak{h}(X)=\bigoplus_{n=0}^{2d} \mathfrak{h}^n(X)\] 
such that $R(\mathfrak{h}^n(X))=H^n(X)$ and $h^n(X)=h^{2d-n}(X)^{\vee}(-d).$
\end{conj}

\begin{conj}\label{conjlef} (Standard conjecture of Lefschetz type)

\noindent For all $n \leq d$, there is an isomorphism
\[\mathfrak{h}^n(X) \isocan \mathfrak{h}^{2d-n}(X)(d-n)\]
induced by the cup product with $L^{d-n}$. Moreover, the K\"unneth decomposition can be refined to a  decomposition 
\[\mathfrak{h}^n(X)= \mathfrak{h}^{n,\prim}(X)\oplus \mathfrak{h}^{n-2}(X)(-1)\] 
 which realizes to the primitive decomposition of $H^n(X)$.
\end{conj}

\begin{remark}
By the work of Lieberman and Kleiman the two conjectures above are known for abelian varieties \cite{GKL,Lieb}, see also Theorem \ref{classic}.
\end{remark}
\begin{assumption}\label{kun+lef}
From now on we will assume that $X$ verifies the   Conjectures \ref{conjkun} and \ref{conjlef}.
\end{assumption}
\begin{remark}
Conjectures \ref{conjkun} and \ref{conjlef} induce an isomorphism
\[\mathfrak{h}^n(X) \isocan \mathfrak{h}^{n}(X)^{\vee}(-n).\]
By adjunction this gives a pairing
\[q_n : \mathfrak{h}^n(X) \otimes \mathfrak{h}^{n}(X)  \longrightarrow \one(-n).\]
By construction  the Lefschetz decomposition 
\[\mathfrak{h}^{n }(X)=\mathfrak{h}^{n,\prim}(X) \oplus \mathfrak{h}^{n-2}(X)(-1)\] 
is orthogonal with respect to this pairing.
\end{remark}
\begin{definition}\label{defpairingmot} For $n\leq d$, we define the pairing 
\[\langle \cdot, \cdot \rangle_{n,\mot} : \mathfrak{h}^n(X) \otimes \mathfrak{h}^{n}(X)  \longrightarrow \one(-n)\]
recursively on $n$, by slightly modifying the pairing $q_n$ of the above remark. We impose that  the Lefschetz decomposition 
\[\mathfrak{h}^{n }(X)=\mathfrak{h}^{n,\prim}(X) \oplus \mathfrak{h}^{n-2}(X)(-1)\] 
is still orthogonal, we impose the equality $\langle \cdot, \cdot \rangle_{n,\mot} =\langle \cdot, \cdot \rangle_{n-2,\mot}(-2) $ on $\mathfrak{h}^{n-2}(X)(-1)$  and finally on $\mathfrak{h}^{n,\prim}(X)$ we define \[\langle \cdot, \cdot \rangle_{n,\mot}  = (-1)^{n(n+1)/2}q_n. \]

 \end{definition}
 \begin{remark}\label{rempola}
If $k$ is embedded in $\C$, the Betti realization of $\langle \cdot, \cdot \rangle_{n,\mot}$ is a polarization of the Hodge structure $H^n(X)$.
\end{remark}
\begin{definition}\label{defpairing} We define  the space $\mathcal{Z}^{n,\prim}_{\num}(X)_{\Q} $  and the  pairing  $\langle \cdot, \cdot \rangle_{n}$.
\begin{enumerate}
\item For $n \leq d/2$, we  define the pairing
 \[\langle \cdot, \cdot \rangle_{n} : \mathcal{Z}^{n}_{\hom}(X)_{\Q}  \times  \mathcal{Z}^{n}_{\hom}(X)_{\Q}  \longrightarrow \Q=\End(\one)\]
as follows. For $\alpha, \beta  \in \mathcal{Z}^{n}_{\hom}(X)_{\Q}  = \Hom (\one, \mathfrak{h}^{2n}(X)(n))$ consider 
\[(\alpha \otimes \beta)  \in   \Hom (\one, \mathfrak{h}^{2n}(X) \otimes \mathfrak{h}^{2n}(X)(2n))\] and define
\[\langle \alpha, \beta \rangle_{n} =      \langle \cdot, \cdot \rangle_{2n,\mot} (2n) \circ  (\alpha \otimes \beta).\]

\item For $n \leq d/2$, define  $\mathcal{Z}^{n,\prim}_{\num}(X)_{\Q} \subset \mathcal{Z}^{n}_{\num}(X)_{\Q}$ as
 \[\mathcal{Z}^{n,\prim}_{\num}(X)_{\Q}  = \Hom_{\NUM(k)_\Q}(\one, \mathfrak{h}^{2n,\prim}(X)(n)).\]  We will keep the same notation as in (1) for the induced pairing
 \[\langle \cdot, \cdot \rangle_{n} : \mathcal{Z}^{n,\prim}_{\num}(X)_{\Q}  \times  \mathcal{Z}^{n,\prim}_{\num}(X)_{\Q}  \longrightarrow \Q.\]
\end{enumerate}
 \end{definition}
 \begin{remark}
The  above definition is equivalent to the one given in the introduction, namely 
\[\mathcal{Z}^{n,\prim}_{\num}(X)_{\Q} = \{\alpha \in \mathcal{Z}^{n}_{\num}(X)_{\Q} , \hspace{0.2cm} \alpha \cdot L^ {{d-2n+1}}=0 \hspace{0.2cm}  \textrm{in}  \hspace{0.2cm}  \mathcal{Z}^{d-n+1}_{\num}(X)_{\Q}\}.\]
Moreover,  for $\alpha$ and $\beta$ in $\mathcal{Z}^{n,\prim}_{\num}(X)_{\Q}$,  we have the equality \[\langle \alpha, \beta \rangle_{n} =  (-1)^n \alpha \cdot \beta \cdot L^{d-2n}.\]
 Notice that in the introduction we
 did not make the Assumption \ref{kun+lef}. Instead Definition \ref{defpairingmot} cannot be formulated without the  Assumption \ref{kun+lef}.
\end{remark}
\begin{conj}\label{conjscht} (Standard conjecture of Hodge type)

\noindent For all $n \leq d/2$, the pairing  \[\langle \cdot, \cdot \rangle_{n} : \mathcal{Z}^{n,\prim}_{\num}(X)_{\Q}  \times  \mathcal{Z}^{n,\prim}_{\num}(X)_{\Q}  \longrightarrow \Q\]
of Definition \ref{defpairing} is positive definite.
\end{conj}

\begin{prop}\label{kerpairing}
The kernel of the pairing $\langle \cdot, \cdot \rangle_{n}$  on the vector space $\mathcal{Z}_{\hom}^{n}(X)_{\Q}$ is the space of algebraic classes  which are numerically trivial. The analogous statement holds true for $\mathcal{Z}_{\hom}^{n,\prim}(X)_{\Q}$.
\end{prop}
\begin{proof}

This is a direct consequence of a general fact. Let $T$ be a homological motive and suppose that it is endowed with an isomorphism
$f:T \isocan T^\vee$. Call $q: T \otimes T \rightarrow \one $ the pairing induced by adjunction. In analogy to Definition \ref{defpairing}, $q$ induces a pairing $\langle \cdot , \cdot \rangle$ on the space $\Hom(\one, T).$ First, notice that this pairing can also be described in the following way
\[\langle \alpha,\beta\rangle=\alpha^\vee \circ f \circ \beta.\]
This equality holds for formal reasons in any rigid category (alternatively, for homological motives, one can check it after realization).
Using this equality one has that $\beta$ is in the kernel of the pairing if and only if  $(\alpha^\vee \circ f) \circ \beta=0$ for all $\alpha$. As $f$ is an isomorphism, this is equivalent to the fact that $\gamma \circ \beta=0$ for all $\gamma\in\Hom(T,\one).$ This means precisely that $\beta$ is numerically trivial.

To conclude, notice that, by construction, one can apply this general fact to $T= \mathfrak{h}^{2n}(X)(n)$ or $T= \mathfrak{h}^{2n,\prim}(X)(n)$.
\end{proof}
\begin{corollary}\label{semipos} 
The pairing $\langle \cdot, \cdot \rangle_{n}$  is perfect on $\mathcal{Z}_{\num}^{n}(X)_{\Q}$. Moreover the pairing  is positive definite on  $\mathcal{Z}_{\num}^{n}(X)_{\Q}$ if and only if it is positive semidefinite on $\mathcal{Z}_{\hom}^{n}(X)_{\Q}$. The analogous statements hold true for $\mathcal{Z}_{\num}^{\prim,n}(X)_{\Q}$.
\end{corollary}
\begin{remark}\label{remHR}
If $k$ is embedded in $\C$ and if we work with Betti cohomology (cf. Remark \ref{rempola}), the Hodge--Riemann relations imply that $\langle \cdot, \cdot \rangle_{n}$  is positive definite on $(n,n)$-classes, hence on $\mathcal{Z}_{\hom}^{n}(X)_{\Q}$. In particular the standard conjecture of Hodge type holds true. 

Moreover, homological and numerical equivalence coincide (recall that we work under the Assumption  \ref{kun+lef}). The argument is due to Liebermann  \cite{Lieb} (see also \cite[Corollary 5.4.2.2]{Andmot}) and we recall it here.

If  $\alpha \in \mathcal{Z}_{\hom}^{n}(X)_{\Q}$ is a nonzero class  in codimension $n\leq d/2$ the inequality  $\langle \alpha, \alpha \rangle_{n}>0$ implies that $\alpha$ is not numerically trivial. On the other hand, the iterated intersection product with the hyperplane section $L$ induces an isomorphism $\mathcal{Z}_{\hom}^{n}(X)_{\Q}\cong \mathcal{Z}_{\hom}^{d-n}(X)_{\Q}$. Hence the intersection product $\mathcal{Z}_{\hom}^{n}(X)_{\Q}\times \mathcal{Z}_{\hom}^{d-n}(X)_{\Q}\rightarrow \Q$ is non-degenerate in one variable if and only if it is non-degenerate in the other, which implies that homological and numerical equivalence must coincide also in codimension $n\geq d/2$.

\end{remark}
Together with the case of characteristic zero (see the above remark), the following is the only known result on the standard conjecture of Hodge type \cite[5.3.2.3]{Andmot}.
\begin{theorem}
The pairing\label{segre} $\langle \cdot, \cdot \rangle_{1}$   on $\mathcal{Z}_{\num}^{1}(X)_{\Q}$ is positive definite. Equivalently, the pairing 
\[\mathcal{Z}^1_{\num}(X)_{\Q}\times \mathcal{Z}_{\num}^1(X)_{\Q} \longrightarrow \Q \hspace{1cm} D,D' \mapsto D\cdot D' \cdot L^{d-2}\]
is of signature  $ (s_+;s_-) = (1;  \rho_1-1)$ with
 $\rho_1=\dim \mathcal{Z}_{\num}^1(X)_{\Q}$.
\end{theorem}
\begin{prop}\label{anypolarization}
The standard conjecture of Hodge type for a fourfold $X$  is equivalent to the statement that the intersection product
\[\mathcal{Z}^2_{\num}(X)_{\Q}\times \mathcal{Z}_{\num}^2(X)_{\Q} \longrightarrow \Q\]
is of signature  $ (s_+;s_-) = ( \rho_2 - \rho_1+1;  \rho_1-1)$ with
 $\rho_n=\dim \mathcal{Z}_{\num}^n(X)_{\Q}$.

In particular the conjecture does not depend on the ample divisor $L$ chosen.
\end{prop}

\begin{proof}
By Theorem \ref{segre} the standard conjecture of Hodge type holds true for divisors.  Hence we have to study codimension $2$ cycles.
Consider the decomposition \[\mathcal{Z}^{2}_{\num}(X)_{\Q} = \mathcal{Z}^{2,\prim}_{\num}(X)_{\Q} \oplus L \cdot \mathcal{Z}^{1}_{\num}(X)_{\Q}\] induced by the Lefschetz decomposition \[\mathfrak{h}^{4}(X)=\mathfrak{h}^{4,\prim}(X) \oplus \mathfrak{h}^{2}(X)(-1).\]  It is orthogonal with respect to the intersection product as the Lefschetz decomposition is orthogonal with respect to the cup product (already at homological level).

We claim that the intersection product on  $L \cdot \mathcal{Z}^{1}_{\num}(X)_{\Q}$ is of signature $(1;  \rho_1-1)$. Assuming the claim notice that  the intersection product is positive definite on $\mathcal{Z}^{2,\prim}_{\num}(X)_{\Q}$ if and only if the intersection product of the total space $\mathcal{Z}^2_{\num}(X)_{\Q}$ has the predicted signature.

For the claim, consider $\alpha=L \cdot D$ an element of $L \cdot \mathcal{Z}^{1}_{\num}(X)_{\Q}$. Then we have 
$\alpha\cdot \alpha = D \cdot D \cdot L^2$
hence the claim follows from Theorem \ref{segre}. 
\end{proof}
\begin{prop} \label{reducefinite}
Let $X_0$ be a specialization of $X$. Suppose that $X_0$ as well satisfies the Assumption \ref{kun+lef} (with respect to the specialization of $L$). Then the standard conjecture of Hodge type for $X_0$ implies the standard conjecture of Hodge type for $X$. 
\end{prop}
\begin{proof} Consider the canonical inclusion  $\mathcal{Z}_{\hom}^{n,\prim}(X)_{\Q}\subset \mathcal{Z}_{\hom}^{n,\prim}(X_0)_{\Q}$ induced by smooth proper base change in $\ell$-adic cohomology. Then, if the pairing $\langle \cdot, \cdot \rangle_{n}$  is positive semidefinite  on $\mathcal{Z}_{\hom}^{n,\prim}(X_0)_{\Q}$ it must be so for  $\mathcal{Z}_{\hom}^{n,\prim}(X)_{\Q} $. On the other hand, by Corollary \ref{semipos}, this semipositivity property is a reformulation of the standard conjecture of Hodge type.
\end{proof}
\begin{remark}
The above proposition reduces the standard conjecture of Hodge type to varieties defined over a finite field. Such reduction step is classical, see for example \cite[Remark 5.3.2.2(2)]{Andmot}.
\end{remark}

The following is our main result. It is shown at the end of Section 8.
\begin{theorem}\label{thmscht}
The standard conjecture of Hodge type holds for abelian fourfolds in positive characteristic.
\end{theorem}

\begin{corollary}\label{thmhomnum}
Let $A$ and $A_0$ be two   abelian fourfolds and suppose that  $A_0$ is a specialization of $A$. Let us fix a prime number $\ell$. If  $\ell$-adic homological equivalence on $A_0$ coincide with numerical equivalence on $A_0$ then the same holds true for $A$.
\end{corollary}
\begin{proof} First, notice that the question whether homological and numerical equivalence coincide only matters for $\mathcal{Z}_{\hom}^{2,\prim}(A)_{\Q}$. Indeed, homological and numerical equivalence coincide for divisors \cite{matsu} (see also \cite[Proposition 3.4.6.1]{Andmot}). This implies   that they coincide also on dimension one cycles,   as the standard conjecture of Lefschetz type holds true for abelian varieties. Finally, consider the decomposition \[\mathcal{Z}_{\hom}^{2 }(A)_{\Q} =\mathcal{Z}_{\hom}^{2,\prim}(A)_{\Q} \oplus L \cdot \mathcal{Z}_{\hom}^{1}(A)_{\Q},\]
and notice that on the complement of $\mathcal{Z}_{\hom}^{2,\prim}(A)_{\Q} $ the equivalences again coincide, as a consequence of the case of divisors.

Now, by smooth proper base change we have $\mathcal{Z}_{\hom}^{2,\prim}(A)_{\Q}\subset \mathcal{Z}_{\hom}^{2,\prim}(A_0)_{\Q}$. If    homological and numerical equivalence  coincide on $A_0$ then the pairing $\langle \cdot, \cdot \rangle_{2}$ on  $\mathcal{Z}_{\hom}^{2,\prim}(A_0)_{\Q}$ is positive definite by Theorem \ref{thmscht}. Hence it is also positive definite on $\mathcal{Z}_{\hom}^{2,\prim}(A)_{\Q}$.   By  Proposition \ref{kerpairing}, this means that there are no nonzero algebraic classes in $\mathcal{Z}_{\hom}^{2,\prim}(A)_{\Q}$ which are numerically trivial.
\end{proof}
\begin{theorem}\label{thmclozel}
Let $A$ be an   abelian fourfold. Then numerical equivalence on $A$ coincides with $\ell$-adic homological equivalence on $A$ for infinitely many prime numbers $\ell$. 
\end{theorem}
\begin{proof}
When $A$ is defined over a finite field, this result is due to Clozel \cite{Clozel}. (Clozel's result actually holds true without the dimensional restriction.) We can reduce to the finite field case by Corollary \ref{thmhomnum}.
\end{proof}
\begin{remark}
The fact that homological and numerical equivalence should always coincide is also one of the four standard conjectures. The two others, namely K\"unneth and Lefschetz, being known for abelian varieties, Theorems \ref{thmscht} and \ref{thmclozel} imply that in order to fully understand the standard conjectures for abelian fourfolds what is missing is $\ell$-independency of $\ell$-adic homological equivalence.
\end{remark}

\section{The motive of an abelian variety}\label{reminder}



In this section we recall classical results on motives of abelian type. We will work with the category of Chow motives $\CHM(k)_{F}$, we fix a Weil cohomology $H^*$ together with its realization functor $R$. Generalities on this category   can be found in the  Conventions. 
\begin{theorem}\label{classic} Let $A$ be an abelian variety of dimension  $g$. Let $\End(A)$ be its ring of endomorphisms (as group scheme) and $\mathfrak{h}(A)\in \CHM(k)_{F}$ be its motive. Then the following holds:
\begin{enumerate}
\item\label{kunneth} \cite{DeMu} The motive $\mathfrak{h}(A)$ admits a Chow--K\"unneth decomposition 
\[\mathfrak{h}(A)=\bigoplus_{n=0}^{2g} \mathfrak{h}^n(A)\] natural in $\End(A)$ and such that 
\[R(\mathfrak{h}^n(A))= H^n(A).\] 
\item\label{kunn} \cite{Ku1} The intersection product induces a canonical isomorphism of graded algebras \[\mathfrak{h}^*(A)=\Sym^* \mathfrak{h}^1(A).\]
\item\label{kings}   \cite[Proposition 2.2.1]{Ki}  The action of $\End(A)$ on $\mathfrak{h}^1(A)$ (coming from naturality in (\ref{kunneth}))  induces an injective morphism of algebras  \[\End(A)\otimes_{\Z}F \hookrightarrow \End_{\CHM(k)_{F}}(\mathfrak{h}^1(A))\]
and if $A$ is isogenous to $B\times C$ then $\mathfrak{h}^1(A) =  \mathfrak{h}^1(B) \oplus \mathfrak{h}^1(C) .$
\item\label{polarization}   \cite{Ku2} The Chow--Lefschetz conjecture holds true for $A$. In particular, the classical isomorphism in $\ell$-adic cohomology  induced by a polarization $H_{\ell}^1(A)\cong H_{\ell}^1(A)^{\vee}(-1)$ lifts to an isomorphism
 \[\mathfrak{h}^1(A)\cong\mathfrak{h}^1(A)^{\vee}(-1).\] 
\end{enumerate}
\end{theorem}
\begin{remark}\label{AHP}
We will need the above results also  in a slightly more general context, namely over the ring of integers of a $p$-adic field. Nowadays these results are known over very general bases, see \cite[Theorem 5.1.6]{OS2} or \cite[Theorem 3.3]{AHP}.
\end{remark}
\begin{definition}\label{defrank}
 A motive is called of abelian type if it  is a  direct factor of the motive of an abelian variety (up to Tate twist).

We say that a motive of abelian type $T$ is  of rank $d$  if the  cohomology groups of $R(T)$ are all zero except in one degree and in that degree the cohomology group is of dimension $d$.  In this case we will write $\dim T = d.$
\end{definition}

\begin{remark}\label{remarkrank}
For motives of abelian type this definition is known to be independent of $R$, see for example  \cite[Corollary 3.5]{jannsen} and \cite[Corollary 1.6]{Ascona}.
\end{remark}
\begin{prop}\label{motivisupersingolari} 
Let $T$ be a motive of abelian type of dimension $d$. 
Consider   its space of algebraic classes modulo numerical equivalence \[V_Z = V_Z(T) =  \Hom_{\NUM(k)_\Q}(\one, T).\] Then the  inequality 
$\dim_\Q V_Z \leq d$ holds.

Moreover, if the equality $\dim_\Q V_Z = d$ holds then we have the following facts:

\begin{enumerate}
\item All  realizations of $T$ are spanned by algebraic classes.
\item Numerical equivalence on $\Hom_{\CHM(k)_\Q}(\one,T)$ coincides with  homological equivalence (for all  cohomologies).
\item Call $L$ the field of coefficients of the realization $R$, then the equality $V_Z\otimes _\Q L = R(T)$ holds.
\end{enumerate}
\end{prop}
\begin{proof}
Let us consider $n$ elements $\bar{v}_1,\ldots,\bar{v}_n$ which are linearly independent in $V_Z$ and fix $v_1,\ldots,v_n\in \Hom_{\CHM(k)_\Q}(\one,T)$ which are liftings of those. By definition of numerical equivalence, there exist  $n$ elements   $f_1,\ldots,f_n $  in $\Hom_{\CHM(k)_\Q}( T,\one)$ such that $f_i(v_j)=\delta_{ij}$. By applying a realization we get $R(f_i)(R(v_j))=\delta_{ij}$ which implies the inequality ${n \leq \dim R(T)=d.}$

When $n=d$, the argument just above shows that $R(v_1),\ldots,R(v_d)$ form a basis of $\Hom(R(\one),R(T))$, which means that $R(T)$ is spanned by the  algebraic classes $R(v_1)(1),\ldots,R(v_d)(1)$, hence we have (1).

Fix  an element $v\in \Hom_{\CHM(k)_\Q}(\one,T)$ and write its realization in the previous basis $R(v)= \lambda_1 R(v_1)+ \ldots + \lambda_d R(v_d)$. Consider now the composition $f_i(v)$. On the one hand, it is a rational number, on the other hand it equals $\lambda_i$. This means that any algebraic class is a rational combination of the basis $R(v_1)(1),\ldots,R(v_d)(1)$. This implies the points (2) and (3).
\end{proof}

\section{Lefschetz classes}\label{sectionlefschetz}
The standard conjecture of Hodge type is known for divisors (Theorem \ref{segre}). In this section we explain how this implies the standard conjecture of Hodge type   for algebraic classes on abelian varieties that are linear combinations of intersections of divisors. This has been already pointed out by Milne \cite[Remark 3.7]{pola} using a different argument.

Throughout the section $A$ is a polarized abelian variety of dimension $g$. We will work with the motive $\mathfrak{h}^{n}(A)$ from Theorem \ref{classic} and the pairing  $\langle \cdot, \cdot \rangle_{n,\mot} $   from Definition \ref{defpairingmot}. 
\begin{definition}\label{defexotic} Let  $\Q[\mathcal{Z}^{1}_{\num}(A)_{\Q}]$ be the subalgebra of $\mathcal{Z}^{*}_{\num}(A)_{\Q}$ generated by $\mathcal{Z}^{1}_{\num}(A)_{\Q}$.  An element of $\Q[\mathcal{Z}^{1}_{\num}(A)_{\Q}]$ is called a  Lefschetz class. The subspace  of Lefschetz classes in $\mathcal{Z}^{n}_{\num}(A)_{\Q}$ is denoted by  $\mathcal{L}^{n}(A)$.
\end{definition}
\begin{remark}\label{nonsisa}
Given two positive integers $a,b$ such that $a\cdot b \leq g$ we have the canonical inclusions
\begin{equation} \label{eq:nonsisa}
\mathfrak{h}^{a\cdot b}(A)=\Sym^{a \cdot b} \mathfrak{h}^1(A) \subset (\Sym^{a} \mathfrak{h}^1(A))^{\otimes b}  = \mathfrak{h}^a(A)^{\otimes b}.
\end{equation}
  induced by Theorem \ref{classic}(\ref{kunn}). In particular any pairing on the right hand side induces a pairing on the left hand side. \end{remark}
\begin{lemma}\label{accoppiamenti1}
Fix an integer $n$ such that   $ 2 \leq 2n\leq g.$ Then the three motivic pairings $\langle \cdot, \cdot \rangle_{1,\mot}^{\otimes 2n}$,  $\langle \cdot, \cdot \rangle_{2,\mot}^{\otimes n}$ and $\langle \cdot, \cdot \rangle_{2n,\mot}$ on $\mathfrak{h}^{2n,\prim}(A)\in \NUM(k)_\Q$ coincide up to a positive rational scalar. Moreover, for each of these pairings, the Lefschetz decomposition $\mathfrak{h}^{2n}(A) = \mathfrak{h}^{2n,\prim}(A) \oplus \mathfrak{h}^{2n-2,\prim}(A) (-1)$ is orthogonal.\end{lemma}

\begin{proof}
 Our statement holds true even at homological level and not just numerically. For this, it is enough to check the statement after realization, for instance on the cohomology group $H_\ell^{2n,\prim}(A)$. 

To do this, take the moduli space of polarized abelian varieties (in mixed characteristic, with some fixed level structure). These pairings are defined on  the relative cohomology of the abelian scheme. Recall that the Zariski closure of the monodromy group associated to this relative cohomology is $\GSp_{2g}= \GSp(H_\ell^{1}(A))$ and that $H_\ell^{2n,\prim}(A)\subset  H_\ell^{2n }(A) = \Lambda^{2n}H_\ell^{1}(A)$  is a geometrically irreducible representation. This implies that these pairings coincide up to a scalar, by Schur's lemma. As these pairings are also defined on Betti cohomology this scalar must be rational. Moreover, they are polarizations of the underlying Hodge structure, so this scalar must be positive.

For the orthogonality part, the argument is the same. If, for a fixed pairing, the decomposition was not orthogonal, we would have a nonzero map between $H_\ell^{2n,\prim}(A)$ and $H_\ell^{2n-2}(A)(-1)^\vee$ which would be $\GSp_{2g}$-equivariant. This is impossible again by Schur's lemma.
 \end{proof}
\begin{proposition}\label{oklefschetz}
Fix an integer $n$ such that   $ 2 \leq 2n\leq g.$ The pairings $\langle \cdot, \cdot \rangle_{1,\mot }^{\otimes 2n}$,  $\langle \cdot, \cdot \rangle_{2,\mot}^{\otimes n}$ and $\langle \cdot, \cdot \rangle_{2n,\mot} $  on  $\mathcal{Z}^{n}_{\num}(A)_{\Q}$ are positive definite if and only if anyone of them is so. Moreover, they are positive definite on $\mathcal{L}^{n}(A)$.\end{proposition}
\begin{proof}
By Lemma \ref{accoppiamenti1} the Lefschetz decomposition \[\mathfrak{h}^{2n}(A) = \mathfrak{h}^{2n,\prim}(A) \oplus \mathfrak{h}^{2n-2,\prim}(A) (-1)\] is orthogonal with respect to any of these pairings, so, arguing by induction on $n$, it is enough to check positivity on algebraic classes of $ \mathfrak{h}^{2n,\prim}(A)$. Again by Lemma \ref{accoppiamenti1} the positivity on the primitive part does not depend on the pairing.

The argument just above works also for Lefschetz classes, indeed each component in the Lefschetz decomposition of a Lefschetz class is again a Lefschetz class  \cite[p. 640]{Milnelef}. Hence we can check positivity for one of the pairings, we will do it for $\langle \cdot, \cdot \rangle_{2,\mot}^{\otimes n}$. Notice that, by construction, the restriction of $\langle \cdot, \cdot \rangle_{2,\mot}^{\otimes n}$ to algebraic classes is $\langle \cdot, \cdot \rangle_{1}^{\otimes n}$, see Definition \ref{defpairing}.

Now, by (\ref{eq:nonsisa}), we have the inclusion
 \[\mathcal{Z}^n_{\num}(A)_{\Q}=\Hom_{\NUM(k)_\Q}(\one, \mathfrak{h}^{2n}(A)(n)) \subset  \Hom_{\NUM(k)_\Q}(\one,(\mathfrak{h}^2(A)(1))^{\otimes n}).\] 
 Moreover Theorem \ref{classic}(\ref{kunn})   implies the equality
 \[\mathcal{L}^{n}(A)_{\Q} = \mathcal{Z}^n_{\num}(A)_{\Q} \cap \mathcal{Z}^1_{\num}(A)_{\Q}^{\otimes n}\]
 where the intersection is taken inside  $\Hom_{\NUM(k)_\Q}(\one,(\mathfrak{h}^2(A)(1))^{\otimes n})$.

On the other hand, the pairing $\langle \cdot, \cdot \rangle_1$ on \[\mathcal{Z}^1_{\num}(A)_{\Q}= \Hom_{\NUM(k)_\Q}(\one, \mathfrak{h}^2(A)(1))\] 
is positive definite by Theorem \ref{segre}. Hence the pairing $\langle \cdot, \cdot \rangle_1^{\otimes n}$ on \[\mathcal{Z}^1_{\num}(A)^{\otimes n}_{\Q} = \Hom_{\NUM(k)_\Q}(\one, \mathfrak{h}^2(A)(1))^{\otimes n}\] is positive definite as well. In particular its restriction to $\mathcal{L}^{n}(A)_{\Q}$ will also be positive definite. \end{proof}

\section{Abelian varieties over finite fields}\label{finitefields}

We start by recalling classical results on abelian varieties over finite fields and afterwards we draw some consequences on motives and algebraic cycles. The abundance of endomorphisms (due to Tate) will allow to decompose the motive of such an abelian variety in small factors. The main results  are Proposition \ref{talvez} and Corollary \ref{lifting motives} which will be essential to apply Theorem \ref{mainthm} to abelian varieties. The latter  says that such a decomposition lifts to characteristic zero and the former that each factor of the decomposition either does not contain algebraic classes (hence it is not interesting for our problem) either it is spanned by algebraic classes.

Throughout the section, we consider an abelian variety $A$ of dimension $g$ over a finite field $k$. We denote by $\End(A)$ the ring of endomorphisms of $A$ and by $\Frob_A$ the Frobenius of $A$. We write $\End^0(A)$ for $\End(A)\otimes_\Z \Q$,  and $*$ for the Rosati involution on it (induced by a polarization).

\begin{theorem}(Tate)\label{tate} A maximal  commutative $\Q$-subalgebra of $\End^0(A)$ has dimension $2g$ \cite{Tate}. There exist maximal commutative $\Q$-subalgebras of $\End^0(A)$ which are the product of CM fields \cite[Lemma 2]{TateBour}.\end{theorem}
\begin{remark}\label{remfrob}
Any maximal commutative $\Q$-subalgebra of $\End^0(A)$ must contain  $\Frob_A$ as the latter is contained in the center of $\End^0(A)$. 
\end{remark}
\begin{definition}\label{definitionCM} 
When $A$ is a simple abelian variety over $k$ a CM-structure for $A$ is the choice of a maximal CM field in $\End^0(A)$. 

For a general $A$ the choice of a decomposition of $A$  in the isogeny category as a product of simple abelian varieties $A_1\times \cdots\times A_t$ and a  CM-structure $L_i$ for each $A_i$ induce a maximal commutative $\Q$-subalgebra ${B=L_1\times \cdots\times L_t}$ of $\End^0(A)$. An algebra $B$  constructed as above  is called a CM-structure for $A$. If such an algebra is $*$-stable we will say that the CM-structure is $*$-stable.
\end{definition}
\begin{remark}
We will see (Corollary \ref{CMstab}) that given a CM-structure there exists (possibly after a finite extension of the base field $k$ and after isogeny) a polarization for which  the CM-structure is $*$-stable.
\end{remark}

\begin{notation}\label{NotationCM}
Let $B$ be a CM-structure for $A$. We write \[{B=L_1\times \cdots\times L_t}\] as a product of CM fields. Let 
 $L$ be the number field which is the Galois closure (over $\Q$) of the compositum of the fields $L_i$, it is a CM number field as well \cite[Proposition 5.12]{goro}. Let $\Sigma_i$ be the set of morphisms from $L_i$ to $L$ and   $\Sigma$ be the disjoint union 
 \[\Sigma = \Sigma_1\cup\cdots \cup \Sigma_t.\]
Write $\bar{\cdot}$ for the action on $\Sigma$ induced by composition with  complex conjugation. We will use the same notation for the induced action on subsets of  $\Sigma$.
\end{notation}

\begin{prop}\label{decomposition}
Let $L$ and   $\Sigma$ be as in Notation  \ref{NotationCM}. Then, in $\CHM (k)_{L}$, the motive $\mathfrak{h}^1(A)$ decomposes into a sum of $2g$ motives  
\[\mathfrak{h}^1(A) =\bigoplus_{\sigma \in \Sigma} M_\sigma,\]
where the action of $b\in L_i$ on $M_\sigma$ induced by Theorem \ref{classic}(\ref{kings}) is given by multiplication by $\sigma(b)$ if $\sigma \in \Sigma_i$ and by multiplication by zero otherwise.  Moreover, each motive $ M_\sigma$ is of rank one (in the sense of Definition \ref{defrank}). Finally, if the CM-structure is $*$-stable, the isomorphism $p:  \mathfrak{h}^1(A)\cong\mathfrak{h}^1(A)^{\vee}(-1)$ of Theorem \ref{classic}(\ref{polarization}) restricts to an isomorphism \[M_\sigma\cong M_{\bar{\sigma}}^\vee(-1)\] for all $\sigma\in \Sigma$, and to the zero map 
\[ M_\sigma \stackrel{0}{\longrightarrow} M_{\sigma'}^\vee(-1) \]
for all $\sigma'\neq {\bar{\sigma}}$.
\end{prop}
\begin{proof}
This is \cite[Corollary 3.2]{Ascona}, we recall here the argument. Consider the inclusion $L_1\times \cdots\times L_t  \hookrightarrow  \End^0(A).$ By Theorem \ref{classic}(\ref{kings}), we deduce an inclusion $(\prod_i L_i)\otimes L  \hookrightarrow  \End_{\CHM^{\ab}(k)_{L}}(\mathfrak{h}^1(A)).$ Each projector of the product of fields $(\prod_i L_i)\otimes L \cong \prod_i L^{\Sigma_i}$ defines a factor $M_{\sigma}.$

Let us now consider the last part of the statement. First, it can be checked after realization   \cite[Corollary 1.12]{Ascona}. Secondly, Rosati involution acts as  complex conjugation on CM fields  \cite[pp. 211-212]{MumfordTata}. Then the statement follows by the very definition of Rosati involution.\end{proof}


\begin{proposition}\label{primadecomposizione}
In the setting of Notation \ref{NotationCM} and Proposition \ref{decomposition}  the following holds:

\begin{enumerate}

\item\label{decomposizioneL} In $\CHM (k)_{L}$ the motive $\mathfrak{h}^n(A)$ decomposes into a sum 
\[\mathfrak{h}^n(A) =\bigoplus_{I\subset \Sigma, \vert I \vert = n} M_I,\]
with $M_I= \otimes_{\sigma \in I}M_\sigma$. Each motive $M_I$ is of rank one (in the sense of Definition \ref{defrank}).

Moreover, if the CM-structure is $*$-stable, then the motives $M_I$ and $M_J$ are mutually orthogonal  in $\NUM(k)_L$ with respect to $\langle \cdot, \cdot \rangle_{1,\mot}^{\otimes n}$ (Definition \ref{defpairingmot}) except if $I = \overline{J}$.

\item\label{decomposizioneQ} In $\CHM (k)_{\Q}$ the motive $\mathfrak{h}^n(A)$ decomposes into a sum of motives
\[\mathfrak{h}^n(A) =\bigoplus  \mathcal{M}_I,\]
where $\mathcal{M}_I$ is spanned by the factors of the form $M_{g(I)}$, with $g$ varying in $\Gal(L/\Q).$

Moreover, if the CM-structure is \nobreak{$*$-stable}, this decomposition in $\NUM(k)_\Q$  is orthogonal with respect to $\langle \cdot, \cdot \rangle_{1,\mot}^{\otimes n}$  (Definition \ref{defpairingmot}).

\end{enumerate}
\end{proposition}
\begin{proof}
Part (\ref{decomposizioneQ}) follows from part (\ref{decomposizioneL}). For the latter, Proposition \ref{decomposition} gives the case $n=1$. Using Theorem \ref{classic}(\ref{kunn}) we deduce the result for higher $n$.
\end{proof}

\begin{proposition}\label{talvez}
Let  $\mathcal{M}_I \in \CHM (k)_{\Q}$ be a  direct factor of $\mathfrak{h}^{2n}(A)$ as constructed in Proposition \ref{primadecomposizione}(\ref{decomposizioneQ}). Then the following holds:
\begin{enumerate}
\item\label{generatoalgebrico} If the vector space $\Hom_{\NUM(k)_\Q}(\one, \mathcal{M}_I (n))$ is not zero then 
\[\dim_\Q \Hom_{\NUM(k)_\Q}(\one, \mathcal{M}_I (n)) = \dim \mathcal{M}_I.\]
In this case homological and numerical equivalence coincide on $ \mathcal{M}_I$ and the realizations of $ \mathcal{M}_I$ are spanned by algebraic classes.
\item\label{generatolefschetz}   If the vector space $\Hom_{\NUM(k)_\Q}(\one, \mathcal{M}_I (n))$ contains a nonzero Lefschetz class (Definition \ref{defexotic}), then all algebraic  classes in 
it are Lefschetz.
\end{enumerate}
\end{proposition}
\begin{proof}
As numerical equivalence commutes with  extension of scalars over $\mathbb{Q}$ \cite[Proposition 3.2.7.1]{Andmot}, we have \[\dim_\Q \Hom_{\NUM(k)_\Q}(\one, \mathcal{M}_I (n)) = \dim_L \Hom_{\NUM(k)_L}(\one, \mathcal{M}_I (n)).\]
 On the other hand, by construction of $\mathcal{M}_I$ (see Proposition \ref{primadecomposizione}(\ref{decomposizioneQ})), we have that the space $\Hom_{\NUM(k)_L}(\one,\mathcal{M}_I(n))$ is generated by the spaces $\Hom_{\NUM(k)_L}(\one,M_{g(I)}(n))$, with $g$ varying in $\Gal(L/\Q).$

 Moreover, as $\Gal(L/\Q)$  acts on the space of algebraic cycles modulo numerical equivalence, we have 
 \[\dim_L \Hom_{\NUM(k)_L}(\one,M_{I}(n)) = \dim_L \Hom_{\NUM(k)_L}(\one,M_{g(I)}(n)),\] for all $g\in \Gal(L/\Q).$
 
 Now, by Proposition \ref{primadecomposizione}(\ref{decomposizioneL}), the motives $M_{I} \in \CHM (k)_{L}$ are of rank one hence, by Proposition \ref{motivisupersingolari}, the above  dimension is either zero or one. This implies the equality in part (\ref{generatoalgebrico}). The rest of part (\ref{generatoalgebrico}) follows from Proposition \ref{motivisupersingolari} as well.

 For part (\ref{generatolefschetz}), the proof just goes as part (\ref{generatoalgebrico}) as all the properties of motives and algebraic classes that we used hold true for Lefschetz motives and Lefschetz classes by   \cite[p. 640]{Milnelef}.
\end{proof}
\begin{remark} The previous proposition is not accessible if one replaces numerical with homological equivalence as it is not known that homological equivalence commutes with extension of scalars. One finds the same issues in \cite{Clozel}. This is the crucial reason for which the main results in this paper are in the setting of numerical equivalence.
\end{remark}
\begin{theorem}[{\cite[Theorem 2]{TateBour}}]\label{HondaTate}
 For any CM-structure $B$ for $A$ (Definition \ref{definitionCM}), the pair  $(A,B)$ lifts to characteristic zero. More precisely there exists a $p$-adic field $K$ with ring of integers $W$ whose residue field $k'$ is a finite extension of $k$, and there is an abelian scheme $\mathcal{A}$ over $W$ together with an embedding $B\hookrightarrow \End^0(\mathcal{A})$, such that $\mathcal{A}\times_W k'$ is isogenous to $A \times_k k'$ (and the isogeny is $B$-equivariant).
 \end{theorem}

\begin{corollary}\label{CMstab}
There exists a polarization on the abelian scheme $\mathcal{A}$ for which the CM-structure is $*$-stable
\end{corollary}
\begin{proof}
Consider the decomposition $\mathcal{A}= \mathcal{A}_1\times \cdots \times \mathcal{A}_t  $ in product  CM simple abelian schemes given from the choice of the CM-structure (Definition \ref{definitionCM}). Choose the polarization to be the product of a polarization on each factor. The statement then reduces to the case when $\mathcal{A}$ is simple and it is enough to check it on the generic fiber, hence in characteristic zero. On the other hand the endomorphism algebra of a simple CM abelian variety in characteristic zero cannot be bigger than the CM-structure itself \cite[\S 22]{MumfordTata}, hence the $*$-stability is automatic. 
\end{proof}

\begin{corollary}\label{lifting motives}
The decompositions of $\mathfrak{h}(A \times_k k')$ in Proposition \ref{primadecomposizione}  lift to decompositions of $\mathfrak{h}(\mathcal{A})\in  \CHM (W)$. Moreover, if the CM-structure of $A$ is $*$-stable (and if the polarisation lifts as well) then the orthogonality statement in Proposition \ref{primadecomposizione} holds true in $ \CHM (W)$ as well.
\end{corollary}


\begin{proof}
The proof of Proposition \ref{primadecomposizione}    is a formal combination of Theorem \ref{classic} together with   the CM-structure.
It works also over $W$ because of Remark \ref{AHP} and Theorem \ref{HondaTate}.
\end{proof}

\section{Exotic classes}\label{sectionexotic}
In this section we fix an abelian variety $A$  of dimension four over a finite field $k$ (and we fix an algebraic closure $\bar{k}$ of $k$). After choosing a CM-structure for $A$ (Definition \ref{definitionCM}), Proposition \ref{primadecomposizione}(\ref{decomposizioneQ}) constructs motives ${\mathcal{M}_I \in \CHM (k)_{\Q}}$ which are direct factors of the motive  $\mathfrak{h}^{4}(A)$. Some of them, that we will call \textit{exotic}, are essential in the proof of the standard conjecture of Hodge type for $A$. The main result of the section (Proposition \ref{exoticdim2}) tells us that they are of rank two (in the sense of Definition \ref{defrank}).
\begin{definition}\label{definizione}
Let  $\mathcal{M}_I \in \CHM (k)_{\Q}$ be a  direct factor of $\mathfrak{h}^{4}(A)$ as constructed in Proposition \ref{primadecomposizione}(\ref{decomposizioneQ}). The motive $\mathcal{M}_I $  is called exotic if the  space of algebraic classes $\Hom_{\NUM(k)_\Q}(\one, \mathcal{M}_I (2))$ is nonzero and it does not contain any  nonzero Lefschetz class (Definition \ref{defexotic}).

An element in the $\mathbb{Q}$-vector space $\Hom_{\NUM(k)_\Q}(\one, \mathcal{M}_I (2))$ will be called an exotic class.
\end{definition}
\begin{remark}\label{remarkuguale}
Notice that, by Proposition \ref{talvez}(\ref{generatoalgebrico}), any exotic motive verifies the equality 
\begin{equation} \label{eq:uguale}
\dim_\Q \Hom_{\NUM(k)_\Q}(\one, \mathcal{M}_I (2)) = \dim \mathcal{M}_I.
 \tag{$\#$}
\end{equation}
Examples of exotic classes on abelian fourfolds (and especially of those that cannot be lifted to characteristic zero) will be discussed in  Appendix \ref{geometricexamples}.
\end{remark}
\begin{proposition}\label{exoticdim2}
Suppose that $A$ has dimension four. Then there exist a finite extension $k'$ of $k$ and a  CM-structure for $A\times_k k'$ such that any exotic motive $\mathcal{M}_I \in\CHM (k')_{\Q}$  of $\mathfrak{h}^{4}(A\times_k k')$ has dimension two. 
\end{proposition}

The proof is decomposed in a series of lemmas and will take the rest of the section. We first fix notation and assumptions that will be used for the lemmas below.
\begin{notation}\label{notnotnot}
Let $\Sigma$ be as in Notation \ref{NotationCM}. Consider the decomposition $\mathfrak{h}^1(A) =\bigoplus_{\sigma \in \Sigma} M_\sigma$ from Proposition \ref{decomposition}. (Recall that  Frobenius acts on each of the eight $ M_\sigma$, see Remark \ref{remfrob}.) Let us denote  $\alpha_\sigma \in \overline{\Q}$ the  eigenvalues for the action of  Frobenius on  $ M_\sigma$. We will denote by $\bar{\cdot}$ the action of  complex conjugation on $ \Sigma$ (or on the set of parts of $ \Sigma$, or on $\overline{\Q}$). 
 
\end{notation}
\begin{assumption}\label{assumassum}
We will suppose that all the algebraic classes $\mathcal{Z}^*_\num (A_{\bar{k}})_\Q$ are defined over $k$. \end{assumption}
\begin{remark}\label{assumptionlef}
Note that the assumption above always holds after a finite extension of $k$. Notice also that, under this assumption, a class which becomes Lefschetz after a finite extension of the base field must be already Leftschetz. Hence, a motive which is exotic over the base field will still be exotic after a finite extension.
\end{remark}

\begin{lemma}\label{musica}
Let $q$ be the cardinality of $k$ and let $\mathcal{M}_I$ be an exotic motive. Then we have  the relation 
\begin{equation} \label{eq:1}
\prod_{\sigma \in I}\alpha_\sigma =q^2.
\end{equation}
Moreover, we have the property 
\begin{equation} \label{eq:2}
 \alpha_\sigma \cdot \alpha_\tau \neq q, \, \, \forall \sigma \neq \tau, \,\, \sigma,\tau \in I.
\end{equation} 
\end{lemma}

\begin{proof}
By Proposition \ref{talvez}(\ref{generatoalgebrico}) the $\ell$-adic realization of $\mathcal{M}_I$ is spanned by algebraic cycles hence  Frobenius acts on it by multiplication by $q^2$.  On the other hand, by construction (see  Proposition \ref{primadecomposizione}(\ref{decomposizioneL})),  Frobenius acts on the line $ M_I\subset \mathcal{M}_I$ via the multiplication by $\prod_{\sigma \in I}\alpha_\sigma $. This gives (\ref{eq:1}).   

Suppose now that  (\ref{eq:2}) is not satisfied.    Then $(\ref{eq:1})$ would force the set    $\{\alpha_\sigma\}_{\sigma \in I}$ to be of the form  $\alpha,q/\alpha, \beta, q/\beta$.   Each of the pairs $(\alpha,q/\alpha)$ and $(\beta, q/\beta)$ would correspond to a Frobenius invariant class in $H^2_\ell (A) (1)$. As the Tate conjecture for divisors on abelian varieties is known \cite[Theorem 4]{Tate}, each of these pairs corresponds to the class of a divisor, hence   $M_I$ would contain a Lefschetz class.
\end{proof}

\begin{definition}
A subset $I$ of $\Sigma$ of cardinality $4$ verifying (\ref{eq:1}) and (\ref{eq:2}) is called an exotic subset.
\end{definition}

\begin{remark}\label{completare}
Lemma \ref{musica} implies that the dimension of the space of exotic classes is at most the number of exotic subsets. The  Tate conjecture for $A$ predicts that this inequality should actually be an equality.
\end{remark}

\begin{lemma}\label{lemmaeven} The following holds.
\begin{enumerate}
\item\label{partizione} Let $I$ be an exotic subset, then the pair $(I,\bar{I})$ forms a partition of $\Sigma$.
\item Any  $\Gal(\overline{\Q}/\Q)$-conjugate of an exotic subset is again exotic.
\item Any exotic motive is even dimensional.
  \end{enumerate}
\end{lemma}
\begin{proof}
 Recall that the Weil conjectures imply $\bar{\alpha}=q/\alpha$. 
Then property (\ref{eq:2}) gives the first part. The second part follows from the very definition of exotic subset. 
Altogether complex conjugation acts without fixed points on  the $\Gal(\overline{\Q}/\Q)$-conjugates of $I$, hence the cardinality of the orbit of $I$ under the action of   $\Gal(\overline{\Q}/\Q)$ is even. On the other hand, this cardinality is the dimension of $\mathcal{M}_I$  (see  Proposition \ref{primadecomposizione}), this concludes the last part.
\end{proof}

\begin{lemma}
Under the Assumption \ref{assumassum}, an exotic subset $I$ verifies
\begin{equation} \label{eq:2n}
(\alpha_\sigma \cdot \alpha_\tau)^n \neq q^n, \, \, \forall \sigma \neq \tau, \,\, \sigma,\tau \in I
\end{equation} 
(for all positive integer $n$).
\end{lemma}
\begin{proof}
 If we extend the scalars to  $\mathbb{F}_{q^n}$  the motive $M_I$  will still be exotic (Remark \ref{assumptionlef}).  The eigenvalues of  Frobenius become $\{\alpha_\sigma^n\}$. Hence, by applying  (\ref{eq:2}) over this new base field we deduce (\ref{eq:2n}).
\end{proof}
\begin{lemma}\label{lemmanotwo}
Under the Assumption \ref{assumassum}, two exotic subsets cannot intersect in exactly two elements.
\end{lemma}
\begin{proof}
 Suppose that there are two exotic subsets  $I$ and $J$ whose intersection has cardinality two. Let us call $\alpha,\beta$ the two Frobenius eigenvalues corresponding to  $I-I\cap J$ and $\gamma,\delta$  those corresponding to  $J-I\cap J$.  Property (\ref{eq:1})  gives
 \[\alpha\cdot \beta=\gamma \cdot  \delta\]
and  property (\ref{eq:2}) gives (after reordering if necessary)
\[\alpha= q/ \gamma \,\, , \,\, \beta=  q/\delta.\]
Putting these relations together we have 
$\alpha^2 =q^2/\beta^2.$ This equality gives a contradiction by applying   (\ref{eq:2n})   to $I$ for $n=2$.
\end{proof}

\begin{lemma}\label{lemmaabove}
Under the Assumption \ref{assumassum}, the dimension of the space of exotic classes is zero, two or four. More precisely, there are either  zero, two or four exotic subsets. In case they are two, they are  complex conjugate to each other. In case they are four, they are of the form $I,\bar{I}, J, \bar{J}$, with $I\cap J$ of cardinality three.
 \end{lemma}
\begin{proof}
The statement on exotic classes is implied by the statement on exotic subsets. Indeed, the dimension of the space of exotic classes is even (Lemma \ref{lemmaeven} together with the relation (\ref{eq:uguale}) in Remark \ref{remarkuguale}) and it is  at most the number of exotic subsets (Remark \ref{completare}).

Let us now show the statement on exotic subsets. We have already pointed out in the proof of Lemma \ref{lemmaeven} that  complex conjugation acts without fixed points on the exotic subsets. Suppose now that there are at least four such subsets, call them $I,\bar{I}, J, \bar{J}$. Then one subset among $J$ and $\bar{J}$ intersects $I$ in at least two elements. Without loss of generality we suppose it is $J$. Then by Lemma \ref{lemmanotwo}, $I\cap J$ must be of cardinality three.

If there were more than four exotic subsets, then, by the same arguments there would be an exotic subset $K$ intersecting $I$ in exactly three elements. Then the intersection of $J$ and $K$ would have exactly two elements. This is impossible  by Lemma \ref{lemmanotwo}.
\end{proof}

\begin{lemma}\label{lemmafour}
Suppose that there are four exotic subsets and that the Assumption   \ref{assumassum} is satisfied. Then, after extending $k= \mathbb{F}_{q}$ to its quadratic extension $\mathbb{F}_{q^2}$, the abelian fourfold $A$ becomes isogenous to  $E\times X$ where  $X$ is an abelian threefold and $E$ is a supersingular elliptic curve on which   Frobenius acts as $q\cdot \id $.

\end{lemma}
 
\begin{proof}
We keep notation from Lemma \ref{lemmaabove}. If the element of $I-I\cap J$ is $\sigma$ then the element of  $(J-I\cap J)= J \cap \bar{I}$ must be $\bar{\sigma}$ because of  Lemma \ref{lemmaeven}(\ref{partizione}).  Let $\alpha$ be the Frobenius eigenvalue for the action on $M_\sigma$, then  (\ref{eq:1}) applied to $I$ and $J$ implies $\alpha=\bar{\alpha}$, hence $\alpha^2=q$.

Let us now extend the field of definition to $\mathbb{F}_{q^2}$. Among the eight eigenvalues of  Frobenius we will find $q$, which means that there is a nonzero Frobenius-equivariant map between the Tate module of the supersingular elliptic curve $E$ and the Tate module of $A$. This implies the statement by  \cite[Theorem 4]{Tate}.
\end{proof}

\begin{lemma}
We keep Notation \ref{notnotnot} and Assumption \ref{assumassum}. Consider on the base field $k=\mathbb{F}_{q^2}$   a supersingular elliptic curve $E$  on which   Frobenius acts as $q\cdot \id$.
Let $A$ be an abelian fourfold of the form $ A=E\times X$, let  $\Sigma_X$ and $ \Sigma_E$ be CM-structures for $X$ and $E$ and define  $\Sigma=\Sigma_X \cup \Sigma_E$. Then, any exotic subset $I$ verifies that $I\cap \Sigma_X$ has cardinality three. Moreover,  the motive $\mathcal{M}_{I\cap \Sigma_X} \in\CHM (k)_{\Q}$, direct factor of $\mathfrak{h}^3(X)$ as constructed in Proposition \ref{primadecomposizione}, is of dimension two. 

Finally, if the space of exotic classes on $A$ is four dimensional, it is contained in the four dimensional motive $\mathcal{M}_{I\cap \Sigma_X} \otimes \mathfrak{h}^1(E)$.
\end{lemma}
\begin{proof}
First recall that  the pair $(I,\bar{I})$ forms a partition of $\Sigma$ by Lemma \ref{lemmaeven}(\ref{partizione}). As $ \Sigma_X$ has cardinality six and is stable by the action of complex conjugation one must have that $I\cap \Sigma_X$ has cardinality three.

Consider the subsets $K \subset \Sigma_X$ (of cardinality three) verifying
\begin{equation} \label{eq:4}
\prod_{\sigma \in K} \alpha_\sigma =q^3
\end{equation}
\begin{equation} \label{eq:5}
 \alpha_\sigma \cdot \alpha_\tau \neq q^2, \, \, \forall \sigma \neq \tau, \,\, \sigma,\tau \in K
\end{equation}
and
\begin{equation} \label{eq:5epoi}
 \alpha_\sigma  \neq q,  \,\, \forall \sigma\in K
\end{equation}

Clearly the relations (\ref{eq:1}) and (\ref{eq:2}) for $I$ imply that $K=I\cap \Sigma_B$ verifies (\ref{eq:4}),(\ref{eq:5}) and (\ref{eq:5epoi}). Conversely, the relations (\ref{eq:4}),(\ref{eq:5}) and (\ref{eq:5epoi})  for $K$ imply that $I=K\cup \{\sigma\}$ verifies (\ref{eq:1}) and (\ref{eq:2}) for any $\sigma \in \Sigma_E$. As there are at most four exotic subsets for $A$, there must be at most two subsets of $\Sigma_X$ verifying  (\ref{eq:4}),(\ref{eq:5}) and (\ref{eq:5epoi}).  On the other hand $I\cap \Sigma_X$  and $\bar{I}\cap \Sigma_X$ are two of those, which implies that there are exactly  two subsets of $\Sigma_X$  verifying (\ref{eq:4}),(\ref{eq:5}) and (\ref{eq:5epoi}).
Finally, as the Galois group $\Gal(\bar{\Q}/\Q)$ acts on those subsets then the motive $\mathcal{M}_{I\cap \Sigma_X}$ has dimension two. 
The rest follows from the construction of $\mathcal{M}_{I\cap \Sigma_X}$.
\end{proof}
\begin{lemma}\label{bellaidea}
Consider an abelian fourfold of the form $ A=E\times X$, where  $E$ is a supersingular elliptic curve such that $\dim_\Q\End(E)\otimes_\Z \Q=4$. Fix a CM-structure for $X$ and consider a direct factor of $\mathfrak{h}^3(X)$  of the form $\mathcal{M}_{I} \in\CHM (k)_{\Q}$, as constructed in Proposition \ref{primadecomposizione}. Suppose that $\mathcal{M}_{I} $ has rank two and that  the  motive ${\mathcal{M}_I \otimes \mathfrak{h}^1(E)}$ is spanned by algebraic classes. Then  there exists a CM structure for  $E$ (write $\Sigma_E=\{\sigma,\bar{\sigma}\}$) such that  the motive $\mathcal{M}_I \otimes \mathfrak{h}^1(E)$ decomposes in the category $\CHM (k)_{\Q}$ as the sum of two motives of rank two 
\[\mathcal{M}_I \otimes \mathfrak{h}^1(E)= \mathcal{M}_{I\cup \sigma} \oplus \mathcal{M}_{I\cup \bar{\sigma}}\]
via Proposition  \ref{primadecomposizione}(\ref{decomposizioneQ}).
 \end{lemma}
\begin{proof}

By construction of   $\mathcal{M}_I$ (Proposition  \ref{primadecomposizione}(\ref{decomposizioneQ})) we can find a quadratic number field $F$  such that the motive $\mathcal{M}_{I} $  decomposes in the category $\CHM (k)_{F}$ into the sum 
\[\mathcal{M}_{I} = M_{I} \oplus  M_{\bar{I}}  \] of two motives of rank one. In order to conclude, it is enough to show that also $\mathfrak{h}^1(E)$  decomposes in the category $\CHM (k)_{F}$ into the sum of two motives of dimension one. 

This fact is equivalent to saying that $F$ can be embedded in the division algebra $\End(E)\otimes_\Z \Q.$ As this algebra splits at all primes except infinity and  the characteristic of the base field (call it $p$), we have to show that the embeddings $F \subset \R$ and $F\subset \Q_p$ cannot exist. Indeed, as the degree of $F$ over $\Q$ is the same as the index of the central $\Q$-algebra $\End(E)\otimes_\Z \Q$ (namely two), the inclusion  $F\subset\End(E)\otimes_\Z \Q$ is equivalent to the fact that $\End(E)\otimes_\Z F$ is a matrix algebra. By the Albert--Brauer--Hasse--Noether theorem this can be checked locally. As $\End(E)\otimes_\Z \Q_\ell$ is already a matrix algebra it is enough to check at $p$ and at infinity. Now if $F $ does not embed in $\R$ and $\Q_p$ then $F\otimes_{\Q} \R$ and $F\otimes_{\Q}\Q_p$ are quadratic field extensions of $\R$ and $\Q_p$. As $\End(E)\otimes_\Z \Q$ has invariant $1/2$ at $p$ and infinity it becomes a matrix algebra after tensoring with these quadratic extensions.

We will give two proofs of the non existence of the embeddings $F \subset \R$ and $F\subset \Q_p$.

A first proof uses that these motives lift to characteristic zero (Proposition \ref{lifting motives}). 
Now if $F$ was contained in $\R$ then the Betti realization of the lifting of $M_{I}$ would respect the Hodge symmetry. This is impossible as it is one dimensional and of weight three. Similarly, if $F$ was contained in $\Q_p$ then the Hyodo-Kato realization of $M_{I}$ and $M_{\bar{I}} $ would be two filtered $\varphi$-modules of dimension one with different filtrations (again because the weight of $\mathcal{M}_{I}$ is three). This implies that  
$\det (M_{I}  \otimes \mathfrak{h}^1(E))$ and $\det(M_{\bar{I}}  \otimes \mathfrak{h}^1(E))$ would realize in two filtered $\varphi$-modules of dimension one and different filtrations.  On the other hand   absolute Frobenius acts on both   in the same way (namely by multiplication by $p^4$) because they are spanned by algebraic classes. This implies that at least one of the two filtered $\varphi$-modules is not admissible and concludes the proof as all filtered $\varphi$-modules coming from geometry (in particular filtered $\varphi$-modules that are realizations of motives) must be admissible.

We give an alternative proof, which does not use the lifting to characteristic zero. We write it for $\Q_p$, it works in the same way for $\R$. If $F$ was contained in $\Q_p$ then the motive $M_{I}  \otimes \mathfrak{h}^1(E)$  would live in $\CHM (k)_{\Q_p}$. On such a motive the division algebra $\End(E)\otimes_\Z \Q_p$ acts hence it acts also on the $\Q_p$ vector space spanned by the algebraic classes of the motive. On the other hand, as this motive is spanned by algebraic classes, the division algebra  $\End(E)\otimes_\Z \Q_p$  would act on a $\Q_p$-vector space  of dimension two. This gives a contradiction as such an  action cannot exist.
\end{proof}
\begin{proof}[Proof of Proposition \ref{exoticdim2}] This is a combination of the  last four lemmas.
\end{proof}

\section{Orthogonal motives of rank $2$}
We start by stating our main technical result,  Theorem \ref{mainthm}. The remark and proposition right after will hopefully give some intuitions on the hypothesis of the statement. We conclude the section by showing that Theorem \ref{mainthm} implies  Theorem \ref{thmscht}. The proof of Theorem \ref{mainthm} will take the rest of the paper and will be ended in Section 13.
\begin{theorem}\label{mainthm}
Let $K$ be a $p$-adic field, $W$  its ring of integers and $k$  its residue field. Let us fix an embedding $\sigma : K \hookrightarrow \C$.
Let \[M \in \CHM(W)_\Q\] be a motive in mixed characteristic and consider the motives induced by pullbacks $M_{\vert\C} \in \CHM(\C)_\Q$ and $M_{\vert{k}} \in \CHM(k)_\Q$. Consider $V_B$ and $V_Z$ the two $\Q$-vector spaces defined as
\[V_B = R_B(M_{\vert\C} ) \hspace{0.5cm}
\textrm{and} \hspace{0.5cm}
V_Z = \Hom_{\NUM(k)_\Q}(\one, M_{\vert{k}}).\]

Let $q$ be a quadratic form on $M$, by which we mean a morphism in $ \CHM(W)_\Q$  of the form
$q: \Sym^2 M  \longrightarrow \one.$
 Consider the two $\Q$-quadratic forms induced by $q$ on $V_B$ and $V_Z$ respectively, 
\[q_B=R_B(q)  \hspace{0.5cm} \textrm{and}  \hspace{0.5cm} q_Z(\cdot)= (q_{\vert{k}} \circ \Sym^2(\cdot)).\]
Suppose that the following holds:
\begin{enumerate}

\item\label{assumptiondim2} The two $\Q$-vector spaces $V_B$  and $V_Z$ are of dimension $2$.

\item The pairing $q_B$ on $V_B$ is a polarization of Hodge structures.

\end{enumerate}

Then the quadratic form $q_Z$ is positive definite.

\end{theorem}

\begin{remark}\label{hypothesismainthm}
\begin{enumerate}
\item The main example of such a motive $M$ we have in mind is an exotic motive of an abelian fourfold (see Section \ref{sectionexotic}). Another example coming from geometry is given in Proposition \ref{propfermat}.
\item\label{homvsrat} One can actually make the  hypothesis a little bit more flexible and work with homological motives instead of Chow motives. This will not matter for our application to abelian varieties.
\end{enumerate}
\end{remark}
\begin{prop}\label{tortura}
Under the hypothesis of Theorem \ref{mainthm} the following holds:
\begin{enumerate}
\item The quadratic form $q_Z$   on $V_Z$ is non-degenerate. 
\item  Given a classical realization $R$, the vector space $R(M_{\vert{k}})$  is of dimension two.
\item Given a classical realization $R$, the vector space $R(M_{\vert{k}})$ is spanned by algebraic classes.
\item Numerical equivalence on $\Hom_{\CHM(k)_\Q}(\one,M_{\vert{k}})$ coincides with the homological equivalence for any classical cohomology.
\item\label{giorno} Fix a classical realization $R$ and call $L$ the field of coefficients of  $R$. Then the equality $V_Z\otimes _\Q L = R(M)$ holds.

\end{enumerate}
\end{prop}
\begin{proof}
Let us start with (2). By the comparison theorem between singular and $\ell$-adic cohomology, we have ${\dim_\Q V_B = \dim_{\Q_\ell}R_\ell(M_{\vert{\C}})}$, for all primes $\ell$, including $\ell=p$.  Then, by smooth proper base change we have   $\dim_{\Q_\ell}R_\ell(M_{\vert{\C}})=\dim_{\Q_\ell}R_\ell(M_{\vert{k}})$ for all $\ell\neq p$ and finally by the $p$-adic comparison theorem we have
 $\dim_{\Q_p}R_p(M_{\vert{\C}})=\dim_{\Frac(W(k))}R_\crys(M_{\vert{k}})$. This concludes (2) as, by hypothesis, $\dim_\Q V_B=2$. 
 
 The proof of (3)-(5) goes as in Proposition \ref{motivisupersingolari}. There, the hypothesis that the motive was of abelian type\footnote{For our main application the motive $M$ will be anyway of abelian type.} was used to ensure that all  realizations have the same dimension (see Remark \ref{remarkrank}). Here, this is replaced by part (2).
 
 Let us now show part (1). First notice that it is enough to show that $R(q_{\vert{k}})$ is non-degenerate (for some classical realization) because of parts (2)-(5).
 Then, again by using the comparison theorems, this is equivalent to the fact that $R_B(q)$ is non-degenerate. On the other hand this is the case as $R_B(q)$ is a polarization of Hodge structures.
\end{proof}
\begin{proof}[Proof of Theorem \ref{thmscht}]
Let $A$ be an abelian fourfold and $L$ be a hyperplane section. By Proposition \ref{reducefinite} we can suppose that $A$ and $L$ are defined over a finite field $k=\mathbb{F}_q$. Notice that, in order to prove Theorem \ref{thmscht}, we can (and will) replace $k$ by a finite extension.

Consider now the decomposition from Proposition \ref{primadecomposizione}(\ref{decomposizioneQ}) and, among the factors of this decomposition, consider 
those that are exotic (Definition \ref{defexotic}). Then there exist a finite extension $k'$ of $k$ and a  CM-structure for $A\times_k k'$ such that any exotic motive  of $\mathfrak{h}^{4}(A\times_k k')$ has dimension two (Proposition \ref{exoticdim2}).  By Theorem \ref{HondaTate}, the CM-structure can be lifted    to characteristic zero.  Moreover, by Proposition \ref{anypolarization}, it is enough to work with a single $L$ for a given abelian fourfold. By Corollary \ref{CMstab} we can choose such an $L$ so that it lifts to characteristic zero and the CM-structure is $*$-stable (Definition \ref{definitionCM}).

Now, by Lemma \ref{accoppiamenti1} we can work with the pairing $\langle \cdot, \cdot \rangle_{1,\mot}^{\otimes 4}$ instead of $\langle \cdot, \cdot \rangle_{4,\mot}$. For this pairing, the decomposition from Proposition \ref{primadecomposizione}(\ref{decomposizioneQ}) is orthogonal, hence we can work with a single motive of the decomposition. By Propositions \ref{oklefschetz} and \ref{talvez}(\ref{generatolefschetz}) we are reduced to motives that are exotic.

Finally, those exotic motives $\mathcal{M}_I$ are settled by Theorem \ref{mainthm} by setting $M=\mathcal{M}_I(2)$. Notice that all the hypothesis of Theorem \ref{mainthm} are satisfied. Indeed, the motive lives in mixed characteristic (Corollary \ref{lifting motives}), together with its quadratic form (because $L$ lifts to characteristic zero). The space $V_B$ is clearly of dimension two; so it is $V_Z$ by   Proposition  \ref{talvez}(\ref{generatoalgebrico}).  Last, the quadratic form $q_B=R_B(\langle \cdot, \cdot \rangle_{1,\mot}^{\otimes 4})$ is a polarization as $R_B(\langle \cdot, \cdot \rangle_{1,\mot})$ is so.
\end{proof}

\section{Quadratic forms}
We recall here some classical facts on quadratic forms. They will allow us to reduce Theorem \ref{mainthm} to a $p$-adic question (Proposition \ref{padicreduction}). For simplicity, we will  work only in the context we will need later, namely with non-degenerate \nobreak{$\Q$-quadratic} forms of rank $2$.
In what follows $\Q_\nu$ denotes the completion of $\Q$ at the place $\nu$.
\begin{definition}\label{Hilbert}
Let $q$ be a $\Q$-quadratic form of rank $2$. Define $\varepsilon_\nu(q)$, the Hilbert symbol of $q$ at $\nu$, as  $+1$ if the equation
\[x^2-q(y,z)=0\]
has a nonzero solution in $x,y,z\in \Q_\nu$, and as $-1$ otherwise. Depending on the context we may write $\varepsilon_p(q)$ or $\varepsilon_\R(q)$.
\end{definition}

\begin{remark}\label{remarksignature}
Let $q$ be a $\Q$-quadratic form of rank $2$. It is positive definite if and only if its discriminant is positive and $\varepsilon_\R(q)=+1$
and it is negative definite if and only if its discriminant is positive and $\varepsilon_\R(q)=-1.$
\end{remark}

\begin{prop}[{\cite[\S 2.3]{cours}}]\label{localinvariant} 
Let $p$ be a prime number and $q_1$ and $q_2$   two non-degenerate $\Q$-quadratic forms of rank $2$. Then 
\[q_1 \otimes \Q_p \cong q_2 \otimes \Q_p\]
if and only if the discriminants of $q_1$ and $q_2$ coincide in $\Q_p^*/(\Q_p^*)^2$ and 
\[\varepsilon_p(q_1)=\varepsilon_p(q_2).\]
\end{prop}

\begin{theorem}[{\cite[\S 3.1]{cours}}]\label{productformula} 
Let $q$ be non-degenerate $\Q$-quadratic form of rank $2$ and $\varepsilon_\nu(q)$ be as in Definition \ref{Hilbert}. Then for all but finite places $\nu$ the equality $\varepsilon_\nu(q)=+1$ holds. Moreover, the following product formula running on all places holds
\[\prod_\nu  \varepsilon_\nu(q)= +1.\]
\end{theorem}

\begin{corollary}\label{signaturetop}
Let $q_1$ and $q_2$  be two non-degenerate $\Q$-quadratic forms of rank $2$ and let $p$ be a prime number. Suppose that, for all primes $\ell$ different from $p$, we have
\[q_1 \otimes \Q_\ell \cong q_2 \otimes \Q_\ell.\]
Then $q_1$ is positive definite if and only if one of the following two cases happens:
\begin{enumerate}
\item  The quadratic forms $q_1\otimes \Q_p$ and $q_2\otimes \Q_p$ are isomorphic and $q_2$ is positive definite.
\item  The quadratic forms $q_1\otimes \Q_p$ and $q_2\otimes \Q_p$ are not isomorphic and $q_2$ is negative definite.
\end{enumerate}
\end{corollary}
\begin{proof}

The $\ell$-adic hypothesis implies in particular that the discriminants of $q_1$ and $q_2$ coincide in $ \Q_\ell^*/(\Q_\ell^*)^2$ 
 for all   $\ell\neq p$. This implies that they coincide in $ \Q^*/(\Q^*)^2$
by \cite[Theorem 3 in 5.2]{Ireland}.

If the discriminants are negative, none of the conditions in the statement holds and the equivalence is clear. From now on we suppose that the discriminants are   positive. By Remark \ref{remarksignature}, $q_1$ is positive definite if and only $\varepsilon_\R(q_1)=+1.$

There are then two cases, namely $q_2$ is positive definite, respectively $q_2$ is negative definite. Again by Remark \ref{remarksignature} they are equivalent to $\varepsilon_\R(q_2)=+1$, respectively $\varepsilon_\R(q_2)=-1$. 

\

Now, Theorem \ref{productformula} implies that
$\prod_\nu  \varepsilon_\nu(q_1)=\prod_\nu  \varepsilon_\nu(q_2).$
Combining this with the $\ell$-adic isomorphisms we deduce
\[\varepsilon_\R(q_1)\varepsilon_p(q_1)=\varepsilon_\R(q_2)\varepsilon_p(q_2).\] This means that $q_1$ is positive definite if and only if
 \[\varepsilon_p(q_1)=\varepsilon_\R(q_2)\varepsilon_p(q_2).\]
 This relation is equivalent to the fact that one of the following two situations hold:
 \begin{enumerate}
\item   $q_2$ is positive definite and $\varepsilon_p(q_1)=\varepsilon_p(q_2).$
\item  $q_2$ is negative definite and $\varepsilon_p(q_1)\neq\varepsilon_p(q_2).$
\end{enumerate}
As the discriminant of $q_1$ and $q_2$ coincide, the equality $\varepsilon_p(q_1)=\varepsilon_p(q_2)$ is equivalent to the fact that $q_1\otimes \Q_p$ and $q_2\otimes \Q_p$ are isomorphic (Proposition \ref{localinvariant}).
\end{proof}

\begin{prop}\label{padicreduction}
Let us keep notation from Theorem \ref{mainthm}. Let $p$ be the characteristic of $k$ and $i$ be the unique non-negative integer such that the Hodge structure $V_B$ is of type $(i,-i)$ and  $(-i,i)$. Then, the quadratic form $q_Z$ is positive definite if and only if the  following   holds:
\begin{enumerate}
\item When $i$ is even: the quadratic forms $q_B\otimes \Q_p$ and $q_Z\otimes \Q_p$ are isomorphic.
\item  When $i$ is odd: the quadratic forms $q_B\otimes \Q_p$ and $q_Z\otimes \Q_p$ are not isomorphic.
\end{enumerate}
\end{prop}
\begin{proof}
By  Proposition \ref{tortura}(\ref{giorno}), we have that $q_Z \otimes \Q_\ell = R_\ell(q_{\vert{k}})$. By the comparison theorem, we have $q_B \otimes \Q_\ell = R_\ell(q_{\vert{\C}})$. Combining these equalities with smooth proper base change in $\ell$-adic cohomology   we deduce that \[q_B \otimes \Q_\ell = q_Z \otimes \Q_\ell.\]
On the other hand, by hypothesis, $q_B$ is a polarization for the Hodge structure $V_B$ hence it is positive definite if $i$ is even and negative definite if $i$ is odd. We can now conclude by applying Corollary \ref{signaturetop} to $q_1=q_Z$ and $q_2=q_B$.
\end{proof}
\section{The $p$-adic comparison theorem}
We keep notation from Theorem \ref{mainthm} and Proposition \ref{padicreduction}.  
The aim of this and next sections is to compare the two $\Q_p$-quadratic spaces of rank two
\[ (V_{B,p},q_{B,p}) \coloneqq (V_B,q_B)\otimes_\Q \Q_p \hspace{0.5cm}  \textrm{and}\hspace{0.5cm}  (V_{Z,p},q_{Z,p}) \coloneqq (V_Z,q_Z)\otimes_\Q \Q_p\]
and deduce from this study Theorem \ref{mainthm}  (via Proposition \ref{padicreduction}). 

The core of the proof is in the next sections, we give here some preliminary results. We start by recalling the $p$-adic comparison theorem (Theorem \ref{tsuji}) and then we apply it to our geometric situation.  

\

Recall that $K$ denotes the $p$-adic field over which the motive $M$ from Theorem \ref{mainthm} is defined.

\begin{theorem}\label{javier} \cite{Fonfon}
There are two integral $\Q_p$-algebras \[B_\crys \subset B_\dR\] the first endowed with  actions of the Galois group $\Gal_K$ and of the absolute Frobenius $\varphi$ and the second endowed with a discrete valuation,  
\[\nu: B_\dR \longrightarrow \Z \cup \{+\infty\}\]
hence in particular endowed with a decreasing filtration \[\Fil^i B_\dR=\{x \in B_\dR, \nu(x)\geq i\}\] verifying 
\[\Fil^i \cdot \Fil^j \subset \Fil^{i+j}.\] 
The $\Q_p$-algebra $ B_\dR$ contains $\overline{\Q}_p$, an algebraic closure of $\Q_p$, and the intersection $ B_\crys \cap \overline{\Q}_p$ is the biggest non-ramified extension of $\Q_p$ inside $\overline{\Q}_p$.

Finally, the following equality holds \cite[Theorem 5.3.7]{Fontp}
 \begin{equation}\label{eq:fontaine}(B_\crys^{\varphi=\id} \cap \Fil^0B_\dR)=\Q_p.\end{equation}
\end{theorem}

\begin{theorem}[\cite{FontMess,Falt,CF}]\label{tsuji} 
There is an  equivalence of $\Q_p$-linear rigid categories 
\[D : \REP \longrightarrow \MOD\]
between the category $\REP$  of crystalline  $\Gal_K$-representations and the category $\MOD$ of admissible filtered $\varphi$-modules (see Convention (\ref{Hyodo-Kato})). This equivalence verifies the following properties:
\begin{enumerate}
\item\label{tsujican} There is a canonical identification 
\[ V \otimes B_\crys=  D(V)\otimes B_\crys \]
which is $\Gal_K$-equivariant and $\varphi$-equivariant. Moreover, the induced isomorphism
\[ V \otimes B_\dR=  D(V)\otimes B_\dR \]
respects the filtrations.

\item The functor $D$ is given by
\[V \mapsto D(V)= (V \otimes B_\crys)^{\Gal_K}\]
and its inverse is given by
\[ [W\otimes B_\crys]^{\varphi=\id} \cap  \Fil^0[W \otimes B_\dR] \mapsfrom W.\]

\item The equivalence $D$  is compatible with the two realization functors\footnote{This last fact is already present in the original works in $p$-adic Hodge theory although they are not written down in the motivic language. A reference is 
\cite[Theorem 1.2]{Niziol}, which can be immediately translated in our setting by  \cite[Proposition 4.2.5.1]{Andmot}.   This fact appears also in \cite[3.4.5]{Andmot} and it is implicitly  used in   \cite[7.4.1]{Andmot}. Nowadays much more general results are known, for instance this comparison of realization functors is available for mixed motives (i.e. motives of varieties which are not necessarely smooth and projective), see \cite[4.15]{NiziFred}.}
$R_p ((\cdot)_{\vert K}): \CHM(W) \rightarrow \REP$   and $R_\HK (\cdot): \CHM(W)\rightarrow \MOD$, namely we have
\[R_\HK (\cdot) = D \circ R_p ((\cdot)_{\vert K}).\]
\end{enumerate}
\end{theorem} 
\begin{corollary}\label{cortsuji}
There are canonical identifications (commuting with the extra structures):
\begin{equation} \label{eq:tsuji}
 (V_{B,p},q_{B,p}) =[(V_{Z,p},q_{Z,p})\otimes B_\crys]^{\varphi=\id} \cap  \Fil^0[(V_{Z,p},q_{Z,p})\otimes B_\dR].
\end{equation}
\begin{equation}\label{chissa}
 (V_{B,p},q_{B,p}) \otimes B_\crys=  (V_{Z,p},q_{Z,p}) \otimes B_\crys.
\end{equation}
\begin{equation}\label{chissa2}
 (V_{B,p},q_{B,p}) \otimes B_\dR=  (V_{Z,p},q_{Z,p}) \otimes B_\dR.
\end{equation}
Moreover, under these identifications,  the  equality of $\Q_p$-algebras  
\begin{equation}\label{stessaalgebra}
 \End_{\Gal_K}(V_{B,p}) = \End_{\varphi,\Fil^*}(V_{Z,p}\otimes \Frac(W(k)))
\end{equation}
holds.
\end{corollary}
\begin{proof}
 Proposition \ref{tortura}(\ref{giorno}) gives the identification
\[  V_{Z,p}\otimes \Frac(W(k)) =  R_\crys (M_{\vert k}).\]
and we have the identification $R_\crys (M_{\vert k})= R_\HK (M)$, see Convention (\ref{Hyodo-Kato}).
On the other hand, we have the equality
\[V_{B,p}  =  R_p (M_{\vert K})\]
because of the comparison theorem between singular and $p$-adic cohomology.
Finally, Theorem \ref{tsuji}(3) implies
\[ D(R_p (M_{\vert K}))   =   R_\HK (M).\]
Altogether we have the equality
\[D(V_{B,p}) = V_{Z,p}\otimes \Frac(W(k)).\]
Notice that these identifications are compatible with the quadratic forms as realization functors and the equivalence $D$ are tensor functors.
Hence Theorem \ref{tsuji}(2) gives  (\ref{eq:tsuji}) and Theorem \ref{tsuji}(1) gives  (\ref{chissa}) and   (\ref{chissa2}).
The identification (\ref{stessaalgebra}) follows from the fact that $D$ is an equivalence of $\Q_p$-linear categories.
\end{proof}

The following proposition settles Theorem \ref{mainthm} in the case where $V_B$ is of type $(0,0)$. It turns out that this case is easier than the others as the quadratic spaces $q_{B,p}$ and $q_{Z,p}$ are not only isomorphic (as predicted by Proposition  \ref{padicreduction}) but also equal (through the identification (\ref{chissa})).

\begin{prop}\label{ordinarycase}
Suppose that the Hodge structure $V_B$ is of type $(0,0)$, then $q_{B,p}$ and $q_{Z,p}$ are isomorphic, hence Theorem \ref{mainthm}  holds true in this case.
\end{prop}
\begin{proof}
First notice that   Frobenius acts trivially on $V_{Z,p}$ because the latter is spanned by  algebraic classes. Second, the hypothesis on the Hodge types gives  
$\Fil^0(V_{Z,p} \otimes K)=V_{Z,p} \otimes K$. Hence the relation (\ref{eq:tsuji}) implies 
\[(V_{B,p},q_{B,p}) =(V_{Z,p},q_{Z,p})\otimes (B_\crys^{\varphi=\id} \cap \Fil^0B_\dR).\]
Using (\ref{eq:fontaine}) we deduce the equality $q_{B,p} = q_{Z,p}$. Theorem \ref{mainthm}  then follows using  Proposition \ref{padicreduction}.
\end{proof}
\begin{remark}\label{boris}\begin{enumerate}
\item This proposition (together with the arguments of the previous sections) gives a full proof of Theorem \ref{thmscht} for ordinary abelian fourfolds. Indeed, in the ordinary case, all Galois invariant classes lift to $(0,0)$-classes as it can be shown with Shimura--Taniyama  formula  \cite[Lemma 5]{TateBour}. See Appendix \ref{geometricexamples} for related discussions.
\item\label{duccio} The hypothesis that $V_B$ is of type $(0,0)$ corresponds to the only case where one can hope that the algebraic classes in $V_Z$ might be lifted to characteristic zero, in which case Theorem \ref{mainthm} would follow from the Hodge--Riemann relations. We find an amusing coincidence that this conjecturally easier  case corresponds to an easier $p$-adic analysis.
\item The proof of the above proposition shows that, under the comparison isomorphisms of Corollary \ref{cortsuji}, one has the equality of $\Q_p$-vector spaces $V_{B,p} = V_{Z,p}$. The Hodge conjecture  predicts   that actually also the two $\Q$-structures $V_B$ and $V_Z$ should coincide as well. Is it possible to show the equality $V_{B} = V_{Z}$ without assuming the Hodge conjecture? We do not know. 

An analogous question can be formulated in the $\ell$-adic setting. Consider an ordinary abelian variety $A$  (of any dimension) together with its canonical lifting $\tilde{A}$. Fix an algebraic class on $A$. Does it corresponds to a Hodge class on $\tilde{A}$? This is a priori weaker than the Hodge conjecture: we do not ask that the algebraic cycle does lift to an algebraic cycle. Notice that if the answer to the question was affirmative then one would have a proof of the standard conjecture of Hodge type for $A$.
\end{enumerate}
\end{remark}

Thank to Proposition \ref{ordinarycase}  we are reduced to the case where $V_B$ is of type $(-i,i),(i,-i),$ for a positive integer $i$.
 It turns out that the only case which is interesting for the application to abelian fourfold is $i=1$, see Remark \ref{altreesotiche}(\ref{altreesotiche2}). We decided to work with a general $i>0$ to keep the possibility of applying  Theorem \ref{mainthm} to varieties other than abelian fourfolds, but the reader might find useful to think of the case $i=1$ in what follows.
 
 \begin{assumption}\label{ipos}
From now on we will suppose that the Hodge structure $V_B$ is not of type $(0,0)$. Equivalently there is a well-defined positive integer $i$ such that $\Fil^i(V_Z\otimes K)$ is a $K$-line and $\Fil^{i+1}(V_Z\otimes K)=0$.
\end{assumption}
 \begin{lemma}\label{isotrop}
 Under Assumption \ref{ipos} the line 
 $\Fil^i(V_Z\otimes K)$ is isotropic.
\end{lemma}
\begin{proof}
As the quadratic form $q$ is motivic its de Rham realization must respect the filtration, so $R_{\dR}(q)(\Fil^i(V_Z\otimes K)) \subset \Fil^{2i}R_{\dR}(\one).$ As $i$ is positive $\Fil^{2i}R_{\dR}(\one) =0.$
\end{proof}

\begin{prop}\label{propendo}
Under  Assumption \ref{ipos} the
$\Q_p$-algebra  $\End_{\Gal_K}(V_{B,p})$
is a field $F$ such that  $[F:\Q_p]=2$. Moreover, for all $v\in V_{B,p}$ and all $f\in F$,  we have the equality
\begin{equation} \label{eq:nonso}
 q_{B,p}(f \cdot v)=N_{F/\Q_p}(f) q_{B,p}(v).
\end{equation}
\end{prop}
\begin{proof}
Consider the ($\Q_p$-points of the) orthogonal  group $G=\SO(q_{B,p})$. We claim that the $\Q_p$-algebra
\[F\coloneqq \Q_p[G]\subset \End (V_{B,p})\]
satisfies all the properties of the statement. First, notice that this algebra is commutative and has dimension two by construction.  

\

As the quadratic form is induced by an algebraic cycle, the Galois group $\Gal_K$ must act on $V_{B,p}$ through $G$. Hence we have the inclusions  
\[\Q_p[\Gal_K]\subset  F \subset  \End (V_{B,p}).\] 
Let us show that $\Q_p[\Gal_K]$ is not of dimension one. If it were so, the algebra $\End_{\Gal_K}(V_{B,p})$ would be isomorphic to $M_{2\times 2}(\Q_p)$. By   (\ref{stessaalgebra}) the algebra $\End_{\varphi,\Fil^*}(V_{Z,p}\otimes \Frac(W(k)))$ would also be $M_{2\times 2}(\Q_p)$, which would imply that $V_{Z,p}\otimes \Frac(W(k))$ would be decomposed into the sum of two isomorphic filtered $\varphi$-modules. This is impossible as it would in particular imply that the filtration on $V_{Z,p}\otimes \Frac(W(k))$ would be a one step filtration, hence contradicting Assumption  \ref{ipos}.

As the $\Q_p$-algebra $\Q_p[\Gal_K]$ is not of dimension one, we deduce that the  inclusion $\Q_p[\Gal_K]\subset  F$ is an equality. The commutator of $F$ being $F$ itself, we also have
$F=\End_{\Gal_K}(V_{B,p})$. 

\

Let us show that the algebra $F$ is not isomorphic to $\Q_p \times \Q_p$. If it were so, arguing as before, we would have a decomposition of 
filtered $\varphi$-modules
$(V_{Z,p}\otimes \Frac(W(k))= W \oplus W'$ and each of the lines $W, W'$ would be isotropic. On the other hand  the line $\Fil^i(V_Z\otimes K)$ is also isotropic (Lemma \ref{isotrop}), hence, we would have  the equality 
\[W\otimes K =\Fil^i(V_Z\otimes K)\]
after possibly replacing $W$ with $W'$.

Now, as $W$ must be admissible, for any nonzero vector $w$ of $W$, the scalar $\alpha$ such that $\varphi(w)=\alpha w$ has $p$-adic valuation equal to $i$.
On the other hand, as  $V_{Z,p}$ is spanned by algebraic classes,  there is a nonzero vector of $W$ which is fixed by $\varphi$. As $i\neq 0$ by Assumption  \ref{ipos}, we deduce a contradiction.

\

As it is not isomorphic to $\Q_p \times \Q_p$ the algebra $F=\Q_p[\SO(q_{B,p})]$ must be a quadratic field extension of $\Q_p.$ The vector space $V_{B,p}$ can be identified with $F$ and the action of $F$ on it can be identified with the (left) multiplication. Under this identification we have $N_{F/\Q_p}(f)= \det(f \cdot)$ and the relation  (\ref{eq:nonso})   follows.
\end{proof}
\begin{remark}
When the motive $M$ from Theorem \ref{mainthm} is  an exotic motive (Section 7), the action of $F$ is induced by algebraic correspondences. The Hodge conjecture predicts that this should always be the case. Indeed, arguing as in the proof of Proposition \ref{propendo}, one can check that $E=\Q[\SO(\Q)]$ acts on $V_B$ respecting the Hodge decomposition and that $F=E \otimes \Q_p.$
\end{remark}
\begin{corollary}\label{nonso}
Keep Assumption  \ref{ipos} and let $F$ be the field of  Proposition \ref{propendo}. Then $F$ acts on $V_{Z,p}$ and the equality  (\ref{eq:tsuji}) is $F$-equivariant. Moreover,  for all $v\in V_{Z,p} $ and all $f\in F$  we have the equality
\begin{equation} \label{eq:nonso2}
q_{Z,p}(f \cdot v)=N_{F/\Q_p}(f) q_{Z,p}(v).
\end{equation}
\end{corollary}
\begin{proof}
If one replaces $V_{Z,p}$ by $V_{Z,p}\otimes \Frac(W(k))$ the statement is a combination of  Corollary \ref{cortsuji}   and Proposition \ref{propendo}. As $F$ commutes with the action of $\varphi$ and $V_{Z,p} \subset V_{Z,p}\otimes \Frac(W(k))$ is precisely the space of vectors which are fixed by $\varphi$, we deduce that $V_{Z,p}$ is stable by $F$ and the statement follows.
\end{proof}
\begin{corollary}\label{norms}
Keep Assumption  \ref{ipos} and let $F$ be the field of  Proposition \ref{propendo}. The following statements are equivalent:
\begin{enumerate}
\item The quadratic forms $q_{B,p}$ and $q_{Z,p}$ are isomorphic.
\item There exists a pair of nonzero vectors  $v_B\in V_{B,p}$ and $v_Z\in V_{Z,p}$ such that $q_{B,p}(v_B)$ and $q_{Z,p}(v_Z)$ are equal in $\Q_p^*/N_{F/\Q_p}(F^*).$
\item For any  pair of nonzero vectors  $v_B\in V_{B,p}$ and $v_Z\in V_{Z,p}$ we have that $q_{B,p}(v_B)$ and $q_{Z,p}(v_Z)$ are equal in $\Q_p^*/N_{F/\Q_p}(F^*).$
\end{enumerate}
\end{corollary}
\begin{proof}
This is a formal consequence of  formulas (\ref{eq:nonso}) and (\ref{eq:nonso2}).
\end{proof}
\begin{remark}\label{conventionF}
By the very construction of $F$, there are two actions of $F$ on  $V_{Z,p}\otimes F$ and this allows us to write the decomposition
\[V_{Z,p}\otimes F= V_{Z,+}  \oplus  V_{Z,-}\]
where $V_{Z,+}$ is the line where the two actions coincide and $V_{Z,-}$ is the line where the two actions are permuted by the non-trivial element of $\Gal(F/\Q_p)$. Using (\ref{eq:nonso2}), we see that $V_{Z,+}$ and $V_{Z,-}$ are also the two isotropic lines of the hyperbolic plane $V_{Z,p}\otimes F$. Notice that there is also an analogous decomposition
\[V_{B,p}\otimes F= V_{B,+} \oplus V_{B,-}\]
with analogous properties and that these decompositions are respected by Corollary \ref{cortsuji}. 

Finally, as $\Fil^i V_{Z,p}\otimes K$ is an isotropic line (Lemma \ref{isotrop}) it must coincide with $V_{Z,+}$ or  $V_{Z,-}$ (after extension of scalars to a field containing $F$ and $K$). We decide that the identification of $F$ and $K$ with a subfield of $\overline{\mathbb{Q}}_p$ is made to have the equality
\[\Fil^i V_{Z,p}\otimes  \overline{\mathbb{Q}}_p = V_{Z,+} \otimes  \overline{\mathbb{Q}}_p.\]  
\end{remark}

\section{Characterization of  $p$-adic periods}\label{p-adic periods}
This section continues the comparison between the two quadratic spaces
$(V_{B,p},q_{B,p})$ and  $(V_{Z,p},q_{Z,p})$ which we initiated in the previous section. We keep notation from there, in particular we work under the Assumption \ref{ipos} which gives a well defined positive integer $i$ and we will make constant use of the field $F$  constructed in  Proposition \ref{propendo}.

The goal is to study the period matrix (relating  the two quadratic spaces $V_{B,p}$ and $V_{Z,p}$) given by Corollary \ref{cortsuji}. 
We compute  the filtration and the action of the Frobenius on these periods. Most importantly, we show that the elements of $B_\crys$ having the same behavior with respect to filtration and Frobenius are essentially unique (this fact characterizes the periods involved).

We   divide the analysis in two cases, depending  whether $F$ is unramified or not.

\subsection*{Unramified case}

In this subsection we work under   Assumption  \ref{ipos} and we assume moreover that the field $F$  constructed in  Proposition \ref{propendo} is unramified over $\Q_p$.
This means that we have the inclusion $F \subset B_\crys$ and that the absolute Frobenius $\varphi$ of $B_\crys$  restricted to $F$ is the  non-trivial element of $\Gal(F/\Q_p)$.
\begin{definition}\label{defpi}
We define the $F$-vector subspace $P_i=P_i(F)$ of $B_\crys$ as the set of $\lambda_i \in B_\crys$ verifying the following properties:
\begin{enumerate}
\item $\varphi^2 (\lambda_i)= \lambda_i.$
\item $\lambda_i \in \Fil ^i B_\dR.$
\item $\varphi(\lambda_i) \in \Fil ^{-i} B_\dR.$
\end{enumerate}
\end{definition}

\begin{proposition}\label{lemmapi}
The $F$-vector space  $P_i(F)$ is of dimension one. Moreover a nonzero element $\lambda_i \in P_i$ verifies the following properties.
\begin{enumerate}
 \item\label{preciso1} $\lambda_i \cdot\varphi(\lambda_i) \in \Q^*_p.$
  \item\label{preciso4} $\lambda_i $ and $\varphi(\lambda_i)$ are invertible in $B_\crys.$
\item\label{preciso2} $\lambda \in \Fil ^i B_\dR-\Fil ^{i+1} B_\dR.$
\item\label{preciso3} $\varphi(\lambda) \in \Fil ^{-i} B_\dR-\Fil ^{-i+1} B_\dR.$
\end{enumerate}
\end{proposition}

\begin{proof}

By construction of $P_i$ we have   $ \lambda_i \cdot\varphi(\lambda_i) \in (B_\crys^{\varphi=\id} \cap \Fil^0B_\dR)$  and on the other hand we have  $(B_\crys^{\varphi=\id} \cap \Fil^0B_\dR)=\Q_p$, see Theorem \ref{javier}. What is missing to show (\ref{preciso1}) is that $\lambda_i \cdot\varphi(\lambda_i)\neq 0.$ It will follow from part (\ref{preciso4}).

  Consider the $\Q_p$-vector space $N=\Q_p^2$ with basis $e_1,e_2$ endowed with 
\begin{equation}\label{biagi}\varphi=\begin{pmatrix}
0 & 1 \\
1 & 0
\end{pmatrix}
\end{equation}
as Frobenius and with the following filtration
\[\Fil^{-i}N=N, \Fil^{-i+1}N =\Fil^i N=\Q_p\cdot e_2, \Fil^{i+1} N=0.  \]
Let us  check that $N$ is an admissible filtered $\varphi$-module.  The action of $\varphi$ on $\det N$, the maximal exterior power of $N$, is by multiplication by $-1$, which has valuation equal to $0$, and the only non trivial step of the induced filtration on $\det N$ is at $0$ as well. There are two lines  stable by $\varphi$ and none of those is $\Q_p\cdot e_2$ hence the only non trivial step of the induced filtration on these lines is at $-i$. On the other hand the action of $\varphi$ on these lines is by multiplication by $\pm 1$, which has valuation equal to $0$. As $i$ is  positive, the required inequality  $-i \leq 0$ holds true.
 
 As $N$ is an admissible $\varphi$-module, by Theorem \ref{tsuji} there exists a unique $\Gal_{\Q_p}$-representation $V$ such that $N=D(V)$. 
 Consider the identification $V \otimes B_\crys= N \otimes B_\crys$ given by  Theorem \ref{tsuji}. It allows to write an element of $V$ with respect to the basis $e_1,e_2$. Let $\begin{pmatrix} \alpha \\
\beta
 \end{pmatrix}$ be a fixed nonzero vector in $V$. As the Frobenius must act trivially on it we get from (\ref{biagi})  the relations
 \[ \alpha=  \varphi(\beta), \,
   \beta= \varphi(\alpha) \]
  which give the equality
  \begin{equation}  \alpha=  \varphi^2(\alpha).\end{equation}
  
    On the other hand, again by Theorem \ref{tsuji}, the vector $\begin{pmatrix} \alpha \\
\beta
 \end{pmatrix}$ lies in $ \Fil^0[M \otimes B_\dR]$. This means that it can be written as a linear combination $\gamma n + \delta e_2$ with $\gamma \in \Fil^{i}B_\dR, \delta \in \Fil^{-i} B_\dR$ and $n \in N$, which implies that $\alpha \in \Fil^{i}B_\dR$ and  $\beta \in \Fil^{-i}B_\dR$. 
 
 Altogether we have    $\alpha \in P_i$ and in particular  $P_i$ is a non empty space. Moreover, arguing as above, one sees that for all $t \in P_i$ the vector $\begin{pmatrix} t   \\
\varphi(t)
 \end{pmatrix}$ is in the vector space $V=[N \otimes B_\crys]^{\varphi=\id} \cap  \Fil^0[N \otimes B_\dR] $. This gives an isomorphism of $\Q_p$-vector spaces between $V$ and $P_i$. As $V$ has dimension $2$ the space $P_i$ is an $F$-vector space of dimension one.
 
 Now notice that  the one dimensional $\Q_p$-vector space of $p$-adic periods relating $\det N$ to $\det V$ is generated by an element in $B_\crys$ of the form $\alpha \beta (f-\varphi( f) )$, for any fixed $f \in F- \Q_p$. As $\det N\otimes B_\crys$ and $\det V\otimes B_\crys$ are free $B_\crys$-modules of rank one the element $\alpha \beta (f-\varphi( f) )$ must be invertible in $B_\crys$, hence $\alpha$ and  $\beta$ are invertible as well.
 
Finally, let us consider the  valuation $\nu$ from Theorem \ref{javier} associated to the filtration. We have showed the inequalities $\nu ( \alpha) \geq +i$ and  $\nu (  \beta)\geq -i$.
On the other hand, as the only non trivial step of the filtration on $\det M$ is  $0$, we deduce that $\nu ( \alpha \beta (f-\sigma( f) ))=0$ and hence $\nu ( \alpha \beta)=0$. This implies  that the above inequalities are actually equalities.
 \end{proof}

\begin{definition}\label{defqi}
Consider $V_{Z,+}$ and $V_{B,+}$ as constructed in Remark \ref{conventionF} and consider them as $F$-vector subspaces of $V_{Z,p}\otimes B_\crys$ via Corollary \ref{cortsuji}. 
We define the $F$-vector subspace $Q=Q(M)$ of $B_\crys$ as the set of $\lambda \in B_\crys$ such that
\[\lambda \cdot V_{B,+} \subset V_{Z,+}.\]
\end{definition}
\begin{remark}
The notation $Q(M)$ would like to suggest that the definition of $Q$ depends a priori on the motive $M$ from Theorem \ref{mainthm}. It will turn out that it actually only depends on $i$ and $F$.
\end{remark}
\begin{proposition}\label{piqi}
The $F$-vector subspaces $P_i,Q\subset B_\crys$ introduced in Definitions \ref{defpi} and \ref{defqi} coincide and they have dimension one.
\end{proposition}

\begin{proof}
 Following Remark \ref{conventionF} we have 
 \begin{equation} \label{eq:mirella0}V_{B,+} \cdot B_\crys =  V_{Z,+} \cdot B_\crys.\end{equation}
  As the  $F$-vector spaces $V_{B,+} $ and $ V_{Z,+}$ have dimension one   then $Q\subset B_\crys$ must have dimension one as well.
 
Fix a nonzero element $\lambda \in Q$. By definition of $Q$ there are nonzero vectors $v_B \in V_{B,+} $ and $v_Z \in V_{Z,+} $ such that
\begin{equation} \label{eq:mirella} \lambda \cdot v_{B} = v_{Z}  \end{equation}
which implies 
\begin{equation} \label{eq:mirella2} \varphi(\lambda) \cdot \varphi(v_{B}) = \varphi(v_{Z}).
\end{equation}
Now  $\varphi$ acts trivially on $V_{B,p}$ and,  because of the presence of algebraic classes, also on $V_{Z,p}$. Hence $\varphi$ acts as the non-trivial element of $\Gal(F/\Q_p)$ on $V_{B,p}\otimes F$ and  $V_{Z,p}\otimes F$. This implies that 
\begin{equation} \label{eq:mirella3} \varphi(v_{B})  \in V_{B,-} \hspace{0.5cm}  \textrm{and} \hspace{0.5cm} \varphi(v_{Z})  \in V_{Z,-} 
\end{equation}
as well as 
\begin{equation} \label{eq:mirella4} \varphi^2(v_B)=v_B \hspace{0.5cm}  \textrm{and} \hspace{0.5cm}   \varphi^2(v_Z)=v_Z.
\end{equation}
We claim that  the above relations (through Corollary \ref{cortsuji}) imply   that $\lambda$ verifies conditions (1),(2) and (3) from Definition \ref{defpi}. Indeed (\ref{eq:mirella}) gives (2), (\ref{eq:mirella2}) and  (\ref{eq:mirella3}) give (3) and finally  (\ref{eq:mirella4}) gives (1).

Hence, we have showed the  inclusion $Q \subset P_i$. For dimensional reasons (Proposition \ref{lemmapi}) it is actually an equality.
\end{proof}

\begin{prop}\label{warmbrinon}
Consider a nonzero element $\lambda_i \in P_i,$ and recall that $\lambda_i \cdot\varphi(\lambda_i) \in \Q^*_p$ (Proposition \ref{lemmapi}).  Then the quadratic forms $q_{B,p}$ and $q_{Z,p}$ are isomorphic if and only if
 \[\lambda_i \cdot\varphi(\lambda_i) \in N_{F/\Q_p}(F^*).\] 
\end{prop}
\begin{proof}
Let $v_B \in V_{B,+}$ and $v_Z \in V_{Z,+}$ as in the proof of Proposition \ref{piqi}.  By (\ref{eq:mirella4}) we have
\begin{equation} \label{eq:mirella9}v_B + \varphi(v_{B})  \in  V_{B,p} \hspace{0.5cm}  \textrm{and} \hspace{0.5cm}    v_Z + \varphi(v_{Z})  \in  V_{Z,p}. 
\end{equation}

Write  $\langle \cdot  ,  \cdot \rangle_{?}$ for the bilinear form associated to the quadratic form $q_?$. Then we have
\[q_{B,p}(v_B + \varphi(v_{B}) ) = 2 \langle v_B ,  \varphi(v_{B}) \rangle_{B,p}  \hspace{0.3cm}  \textrm{and}   \hspace{0.3cm} q_{Z,p}(v_Z + \varphi(v_{Z}) ) = 2 \langle v_Z ,  \varphi(v_{Z}) \rangle_{Z,p}  \]
because    $v_B,v_Z,\varphi(v_B)$ and $\varphi(v_Z)$ are isotropic vectors (see Remark \ref{conventionF} and (\ref{eq:mirella3})). 
Hence, using  relations  (\ref{eq:mirella}), (\ref{eq:mirella2}), together with (\ref{chissa}), we have
\[\lambda_i \cdot\varphi(\lambda_i) \cdot q_{B,p}(v_B + \varphi(v_{B}) ) =  q_{Z,p}(v_Z + \varphi(v_{Z}) ) . \]
We can conclude thanks to Corollary \ref{norms}.
\end{proof}

\subsection*{Ramified case}\label{ramcase}

In this subsection we work under the Assumption  \ref{ipos} and we assume moreover that the field $F$  constructed in  Proposition \ref{propendo} is ramified over $\Q_p$.
Define $\sigma$ to be the  non-trivial element of $\Gal(F/\Q_p)$ and $  B_{\crys,F} \subset B_\dR$ as the smallest subring containing $F$ and  $B_{\crys}$.
The inclusions   $  B_{\crys} \subset B_{\crys,F}$ and  $  F \subset B_{\crys,F}$ induce an identification
\[F\otimes_{\Q_p} B_{\crys}=B_{\crys,F}. \]
This allows  to extend $\sigma$ and $\varphi$ to two endomorphisms of $\Q_p$-algebras
\[   \sigma, \varphi : B_{\crys,F}    \longrightarrow B_{\crys,F} \]
which commute.

\begin{definition}\label{defpiram}
We define the $F$-vector subspace $P_i=P_i(F)$ of $B_{\crys,F}$ as the set of $\lambda \in B_{\crys,F}$ verifying the following properties.
\begin{enumerate}

\item $\lambda \in \Fil ^i B_\dR.$
\item $\sigma(\lambda) \in \Fil ^{-i} B_\dR.$
\item $\varphi (\lambda)= \lambda.$
\end{enumerate}
\end{definition}

\begin{proposition}\label{lemmapiram} The $F$-vector space  $P_i(F)$ is of dimension one. 
Moreover a nonzero element $\lambda_i \in P_i$  verifies the following properties.
\begin{enumerate}
 \item  $\lambda_i \cdot\sigma(\lambda_i) \in \Q^*_p.$
  \item  $\lambda_i$ and $\sigma(\lambda_i)$ are invertible in $B_{\crys,F}.$
\item  $\lambda_i \in \Fil ^i B_\dR-\Fil ^{i+1} B_\dR.$
\item  $\sigma(\lambda_i) \in \Fil ^{-i} B_\dR-\Fil ^{-i+1} B_\dR.$
\end{enumerate}
\end{proposition}

\begin{proof}
By construction of $P_i$ we have that $\lambda_i \cdot\varphi(\lambda_i) $ is invariant under $\sigma$ hence it is in $B_\crys$. Moreover it verifies  $ \lambda_i \cdot\varphi(\lambda_i) \in (B_\crys^{\varphi=\id} \cap \Fil^0B_\dR)$  and on the other hand we have  $(B_\crys^{\varphi=\id} \cap \Fil^0B_\dR)=\Q_p$, see Theorem \ref{javier}. What is missing to show (\ref{preciso1}) is that $\lambda_i \cdot\varphi(\lambda_i)\neq 0.$ It will follow from part (\ref{preciso4}).

 Consider the $\Q_p$-vector space $N=\Q_p^2$ endowed with the identity as Frobenius. Let $L$ be a line in $N\otimes F$ which is not stable under $\sigma$.
Define the following filtration
\[\Fil^{-i}N\otimes F=N\otimes F, \Fil^{-i+1}N\otimes F =\Fil^i N \otimes F=\sigma(L), \Fil^{i+1} N=0.  \]
Let us  check that $N$ is an admissible filtered $\varphi$-module.  The action of $\varphi$ on $\det N$, the maximal exterior power of $N$, is by multiplication by $+1$, which has valuation equal to $0$, and the only non trivial step of the induced filtration on $\det N$ is at $0$ as well. Consider a line $L'$ in $N$. The action of $\varphi$ on it is again by multiplication by $+ 1$, which has valuation equal to $0$, and the only non trivial step of the induced filtration on it is at $-i$ because $L'\neq \sigma(L)$. As $i$ is  positive, the required inequality  $-i \leq 0$ holds true.
 
 As $N$ is an admissible $\varphi$-module, by Theorem \ref{tsuji} there exists a unique $\Gal_{\Q_p}$-representation $V$ such that $N=D(V)$. Consider the action of $f\in F$ on $N\otimes F$ given by multiplication by $f$ on $L$ and by multiplication by $\sigma(f)$ on $ \sigma(L)$. It descends over $\Q_p$ and gives an action of $F$ on $N$. Moreover this action respects the filtration. By  Theorem \ref{tsuji}, the field $F$ acts on $V$ as well and  the identification $V \otimes B_\crys= N \otimes B_\crys$ is $F$-equivariant. In particular there are eigenvectors 
 $m\in L$ and  $v \in V \otimes F$ and a scalar $\alpha \in B_{\crys,F}$ such that 
  \[
\alpha \cdot m=   v
 \]
 and hence $ \sigma(\alpha) \cdot \sigma(m)=   \sigma(v)$.
 
As the Frobenius acts trivially on $V$ and $N$ we deduce that it acts as the identity on  $\alpha$ and $\sigma(\alpha)$ as well. 
Moreover,  the inclusion $V \subset \Fil^0 N\otimes B_\dR$ implies $\alpha \in \Fil ^i B_\dR-\Fil ^{i+1} B_\dR$ and $\sigma(\alpha) \in \Fil ^{-i} B_\dR-\Fil ^{-i+1} B_\dR.$ 

 Altogether we have    $\alpha \in P_i$ and in particular  $P_i$ is a non empty space. Moreover, arguing as above, one sees that for all $t \in P_i$ the linear combination $t \cdot m + \sigma(t) \cdot \sigma (m)$  is in the vector space $V=[N \otimes B_\crys]^{\varphi=\id} \cap  \Fil^0[N \otimes B_\dR] $. This gives an isomorphism of $\Q_p$-vector spaces between $V$ and $P_i$. As $V$ has dimension $2$ the space $P_i$ is an $F$-vector space of dimension one.
 
 Finally, as $v$ and $m$ generates the same free  $B_{\crys,F}$-module of rank one, the scalar $\alpha$ must be invertible in $B_{\crys,F}$, and the same argument applies for $ \sigma(\alpha)$.
\end{proof}

\begin{definition}\label{defqiram}
Consider $V_{Z,+}$ and $V_{B,+}$ as constructed in Remark \ref{conventionF} and consider them as $F$-vector subspaces of $V_{Z,p}\otimes B_{\crys,F}$ via Corollary \ref{cortsuji}. 
We define the $F$-vector subspace $Q=Q(M)$ of $B_{\crys,F}$ as the set of $\lambda \in B_{\crys,F}$ such that
\[\lambda \cdot V_{B,+} \subset V_{Z,+}.\]
\end{definition}
\begin{proposition}\label{piqiram}
The $F$-vector subspaces $P_i,Q\subset B_{\crys,F}$ introduced in Definitions \ref{defpiram} and \ref{defqiram} coincide and they have dimension one.
\end{proposition}
\begin{proof}
Analogous to the proof of Proposition \ref{piqi}.
\end{proof}

\begin{prop}\label{warmbrinonram}
Consider a nonzero element $\lambda_i \in P_i,$ then \[\lambda_i \cdot\sigma(\lambda_i) \in \Q^*_p.\] Moreover the quadratic forms $q_{B,p}$ and $q_{Z,p}$ are isomorphic if and only if
 \[\lambda_i \cdot\sigma(\lambda_i) \in N_{F/\Q_p}(F^*).\] 
\end{prop}
\begin{proof}
Analogous to the proof of Proposition \ref{warmbrinon}.
\end{proof}
\begin{remark}  
For any $i$ and any $F$ the periods $P_i(F)$ appear in the realization of a motive $M_{i,F}$. Indeed, for $i=1$ this comes from
Proposition \ref{propexemple}, see also Remark \ref{bellaciao}(\ref{unamattina}). For a different $i$ one can consider $M_{i,F}=M_{1,F}^{\otimes i}$, where the tensor product is done in the $F$-linear category of motives endowed with an $F$-action. 

Define $P_0(F)=F$ and $P_*(F)=\bigoplus_{i\geq 0} P_i(F)$. By the very definition of $P_i(F)$, the set $P_*(F)$ is a graded $F$-algebra. Moreover Propositions \ref{lemmapi} and \ref{lemmapiram} imply that $P_*(F)$  is non canonically isomorphic to the polynomial algebra $F[x]$.

\end{remark}

\section{Computation of $p$-adic periods}\label{computation}
The aim of this section is to exhibit explicit elements of the space of periods $P_i(F)$ introduced in Definition \ref{defpi} when $F$ is unramified and in Definition \ref{defpiram} when $F$ is ramified. This is done with Propositions  \ref{warmbrinon} and \ref{warmbrinonram} in mind. The construction of these periods is inspired by \cite[\S 9]{Colmez}.

 We will still divide the analysis in two cases, depending  whether $F$ is unramified or not.  Moreover the ramified case will need some extra special cases for $p=2$.

 \subsection*{Unramified case}

In  this subsection we work under   Assumption  \ref{ipos} and we assume moreover that the field $F$  constructed in  Proposition \ref{propendo} is unramified over $\Q_p$.
This means that we have the inclusions \[F \subset  B_{\crys} \subset B_{\dR}\] and that the absolute Frobenius $\varphi$ of $B_\crys$  restricted to $F$ is the  non-trivial element of $\Gal(F/\Q_p)$.

\begin{prop}\label{lubintate} (Lubin--Tate)
There is an element $t_2 \in B_\crys$ which verifies the following properties:
\begin{enumerate}
\item $t_2 \in \Fil ^{1} B_\dR-\Fil ^{2} B_\dR$.
\item  $ \varphi(t_2) \in \Fil ^{0} B_\dR-\Fil ^{1} B_\dR$.
\item $\varphi^2(t_2) = p \cdot t_2$.
\item Both $t_2 $ and  $ \varphi(t_2) $ are invertible in $B_{\crys}$.
\end{enumerate}
\end{prop}
\begin{proof}
This period will appear as the period of an explicit filtered $\varphi$-module. Consider the vector space  $N=\Q_p^2$ with basis $e_1,e_2$ endowed with the Frobenius
\begin{equation}\label{darienzo}  
\varphi =  \begin{pmatrix} 0 & 1/p \\
1 & 0
 \end{pmatrix}.
 \end{equation}
 Define the filtration to be $\Fil^{-1} N=N, \, \Fil^0 N= \Q_p e_2, \, \Fil^{1} N=0.$ This filtered $\varphi$-module is admissible. To check so notice that there are no non trivial submodules of $N$ stable by $\varphi$ hence the admissibility condition has to be checked only on $N$. 
 Moreover the action of $\varphi$ on $\det N$, the maximal exterior power of $N$, is by multiplication by $-1/p$, which has valuation equal to $-1$, and the only non trivial step of the induced filtration on $\det N$ is at $-1$ as well.
 
 As $N$ is an admissible $\varphi$-module, by Theorem \ref{tsuji} there exists a unique $\Gal_{\Q_p}$-representation $V$ such that $N=D(V)$. 
 Consider the identification $V \otimes B_\crys= N \otimes B_\crys$ given by  Theorem \ref{tsuji}. It allows to write an element of $V$ with respect to the basis $e_1,e_2$. Let $\begin{pmatrix} \alpha \\
\beta
 \end{pmatrix}$ be a fixed nonzero vector in $V$. As the Frobenius must act trivially on it we get from (\ref{darienzo})  the relations
 \[ \alpha= 1/p \cdot \varphi(\beta),\,
    \beta= \varphi(\alpha),\]
  which give the equality
  \[ p \cdot \alpha=  \varphi^2(\alpha).\]
  We claim that $t_2= \alpha$ and $\varphi(t_2)= \beta$ verify all the properties of the proposition.
  
  By Theorem \ref{tsuji} the vector $\begin{pmatrix} \alpha \\
\beta
 \end{pmatrix}$ belongs to $\Fil^0[M \otimes B_\dR]$. This means that it can be written as $\gamma n + \delta e_2$ with $\gamma \in \Fil^{1}B_\dR, \delta \in \Fil^0 B_\dR$ and $n \in N$, which implies that $\alpha \in \Fil^{1}B_\dR$ and  $\beta \in \Fil^{0}B_\dR$. 
  
  Now notice that, for all $f\in F$, the vector $\begin{pmatrix} f\cdot \alpha \\
\varphi(f)\beta
 \end{pmatrix}$ is in the vector space $V=[N\otimes B_\crys]^{\varphi=\id} \cap  \Fil^0[N \otimes B_\dR] $. Hence,  the one dimensional $\Q_p$-vector space of $p$-adic periods relating $\det N$ to $\det V$ is generated by an element in $B_\crys$ of the form $\alpha \beta (f-\sigma( f) )$, for any fixed $f \in F- \Q_p$. As $\det N\otimes B_\crys$ and $\det V\otimes B_\crys$ are free $B_\crys$-modules of rank one the element $\alpha \beta (f-\sigma( f) )$ must be invertible in $B_\crys$, hence $\alpha$ and  $\beta$ are invertible as well.
 
Finally, let us consider the  valuation $\nu$ from Theorem \ref{javier} associated to the filtration. We have showed the inequalities $\nu ( \alpha) \geq 1$ and  $\nu (  \beta)\geq 0$.
On the other hand, as the only non trivial step of the filtration on $\det M$ is  $-1$, we deduce that $\nu ( \alpha \beta (f-\sigma( f) ))=1$ and hence $\nu ( \alpha \beta)=1$. This implies  that the above inequalities are actually equalities and concludes the proof.
 \end{proof}
\begin{corollary}\label{brinon}
Fix a positive integer $i$ and define $\lambda_i \in B_\crys$ as
\[\lambda_i = (t_2/\varphi(t_2))^i .\]
Then $\lambda_i$ is a nonzero element of  $P_i$ and moreover it verifies
\[\lambda_i \cdot\varphi(\lambda_i) = 1/p^i.\] 
\end{corollary}
\begin{proof}
All properties follow from Proposition \ref{lubintate} by direct calculation.
\end{proof}

\subsection*{Ramified case}\label{ramcase}

In this subsection we work under the Assumption  \ref{ipos} and we assume that the field $F$  constructed in  Proposition \ref{propendo} is ramified over $\Q_p$. This subsection is written in  analogy to the unramified case, although  some extra computations are needed to deal with the case $p=2$.

Define $\sigma$ to be the  non-trivial element of $\Gal(F/\Q_p)$ and $  B_{\crys,F} \subset B_\dR$ as the smallest subring containing $F$ and  $B_{\crys}$.
The inclusions   $  B_{\crys} \subset B_{\crys,F}$ and  $  F \subset B_{\crys,F}$ induce an identification
\[F\otimes_{\Q_p} B_{\crys}=B_{\crys,F}. \]
This allows to extend $\sigma$ and $\varphi$ to two endomorphisms of $\Q_p$-algebras
\[   \sigma, \varphi : B_{\crys,F}    \longrightarrow B_{\crys,F} \]
which commute.

Finally, we will denote by $\Q_{p^n}\subset B_{\crys}$ the unique unramified field extension of $\Q_p$ of degree $n$.

\begin{prop}\label{lubintateram} (Colmez)
For each uniformizer $\pi \in F$ there exists an element $t_\pi \in B_{\crys,F}$  which verifies the following properties.
\begin{enumerate}
\item $t_\pi \in \Fil ^{1} B_\dR-\Fil ^{2} B_\dR$.
\item $\sigma (t_\pi) \in \Fil ^{0} B_\dR-\Fil ^{1} B_\dR$.
\item  $ \varphi(t_\pi) = \pi  \cdot t_\pi$.
\item Both $t_\pi $ and  $ \sigma(t_\pi) $ are invertible in $B_{\crys,F}$.
\end{enumerate}
\end{prop}
\begin{proof}
This period will appear as the period of an explicit filtered $\varphi$-module. Let $x^2+ax+b$ be the minimal polynomial of $1/\pi$ over $\Q_p$. Consider the vector space  $N=\Q_p^2$ with basis $e_1,e_2$ endowed with the Frobenius
\[ 
\varphi =  \begin{pmatrix} 0 & -b \\
1 & -a
 \end{pmatrix}.
 \]
 Fix an eigenvector $m \in M \otimes F$ satisfying
 \begin{equation}\label{disarli2}  
\varphi (m)=  \frac{1}{\pi} m
 \end{equation}
 and hence 
  $
\varphi (\sigma(m))=  \frac{1}{\sigma(\pi)} \sigma(m). 
 $
 Define the filtration to be \[\Fil^{-1} (N\otimes F)=N\otimes F, \, \Fil^0 (N\otimes F)= F\cdot ( \sigma(v)), \, \Fil^{1} (N\otimes F)=0.\] We claim that this filtered $\varphi$-module is admissible. To check so notice that there are no non trivial submodules of $N$ stable by $\varphi$ hence the admissibility condition has to be checked only on $N$. Moreover the action of $\varphi$ on $\det N$, the maximal exterior power of $N$, is by multiplication by $b$, which has valuation equal to $-1$, as $\pi$ is a uniformizer of $F$, and the only non trivial step of the induced filtration on $\det N$ is at $-1$ as well.
 
  As $N$ is an admissible filtered $\varphi$-module we can apply Theorem \ref{tsuji} and consider the $\Gal_{\Q_p}$-representation $V$ which verifies  $N=D(V)$. Now notice that the action of $\varphi$ on $N$ respects the filtration, this means that the field $F$ acts on $N$ as filtered $\varphi$-module. Then Theorem \ref{tsuji}  implies that
 $F$ acts on $V$ as well and that the canonical isomorphism  $V \otimes B_\crys= N \otimes B_\crys$ is $F$-equivariant. In particular we can find an eigenvector $v \in V \otimes F$ and a scalar $\alpha \in B_{\crys,F}$ such that 
  \begin{equation}\label{disarli4}  
\alpha \cdot m=   v
 \end{equation}
 and hence $ \sigma(\alpha) \cdot \sigma(m)=   \sigma(v)$.
 
As the Frobenius acts trivially on $V$, relations (\ref{disarli2}) and (\ref{disarli4}) imply $ \varphi(\alpha) = \pi  \cdot \alpha$. We claim now that the period $t_\pi=\alpha$ satisfies all the properties of the statement. Indeed, as $v$ and $m$ generates the same free  $B_{\crys,F}$-module of rank one, the scalar $\alpha$ must be invertible in $B_{\crys,F}$, and the same argument applies for $ \sigma(\alpha)$. As the filtration is trivial on $V \otimes F$, relation   (\ref{disarli4}) implies that $\alpha \in \Fil ^{1} B_\dR-\Fil ^{2} B_\dR$ and again the same argument applies for $ \sigma(\alpha)$.
 \end{proof}
\begin{corollary}\label{brinonram}
Suppose that there is a uniformizer $\pi\in F$ such that 
\[\sigma(\pi)=-\pi.\]
Let $\sqrt{a}\in \Q_{p^2}$ be an element of the quadratic unramified extension of $\Q_p$ such that
\[\varphi(\sqrt{a}) = -\sqrt{a}.\]
Define $\lambda_i \in B_{\crys,F}$ as
\[\lambda_i = (\sqrt{a}  \cdot t_\pi/ \sigma(t_\pi))^i .\]
Then $\lambda_i$ is a nonzero element of  $P_i(F)$ and moreover it verifies
\[\lambda_i \cdot\sigma(\lambda_i) = a^i.\] 
\end{corollary}
\begin{proof}
All properties follow from  Proposition \ref{lubintateram} by direct calculation.
\end{proof}

\begin{remark}\label{remtame}
Let us recall the list of quadratic extensions of $\Q_p$. If $p\neq 2$, and if $a$ is an integer which is not a square modulo $p$, then $\Q_p(\sqrt{a})$ is the unramified quadratic extension, whereas $\Q_p(\sqrt{p})$ and $\Q_p(\sqrt{ap})$ are the two ramified quadratic extensions. 
If $p=2$,  $\Q_2(\sqrt{5})$ is the unramified quadratic extension, whereas $\Q_2(\sqrt{2}),\Q_2(\sqrt{6}),\Q_2(\sqrt{10}),\Q_2(\sqrt{14}),\Q_2(\sqrt{3})$ and $\mathbb{Q}_2(\sqrt{-1})$ are the six ramified quadratic extensions. 

In particular, it is almost always possible to find a uniformizer $\pi\in F$ such that 
$\sigma(\pi)=-\pi$ as in  Corollary \ref{brinonram}. The only exceptions are for  $p=2$ and  $F=\mathbb{Q}_2(\sqrt{3})$ or  $F=\mathbb{Q}_2(\sqrt{-1})$.  What follows treats these two exceptions.
\end{remark}

\begin{corollary}\label{brinonwild}
We keep notation from Proposition  \ref{lubintateram} and we work  with $p=2$ and the field  $F=\mathbb{Q}_2(\sqrt{-1})$.  Let $\alpha \in \Q_{2^4}(\sqrt{-1})$ be an element  such that
\[\varphi(\alpha) =  \sqrt{-1} \cdot \alpha.\]
Define $\lambda_i \in B_{\crys,F}$ as
\[\lambda_i = (\alpha \cdot t_{1-\sqrt{-1}} /\sigma(t_{1-\sqrt{-1}}))^i .\]
Then $\lambda_i$ is a nonzero element of  $P_{i}(F) $ and  moreover it verifies
\[\lambda_i \cdot\sigma(\lambda_i) = (\alpha  \cdot \sigma(\alpha) )^i.\] 
\end{corollary}
\begin{proof}
All properties follow from Proposition \ref{lubintateram}  by direct calculation.
\end{proof}

\begin{prop}\label{idearef} (Colmez)
For each uniformizer $\pi \in F$ there exists an element $t_{2,\pi} \in B_{\crys,F}$  which verifies the following properties.
\begin{enumerate}
\item $t_{2,\pi}  \in \Fil ^{1} B_\dR-\Fil ^{2} B_\dR$.
\item $\sigma (t_{2,\pi} ) , \varphi (t_{2,\pi} ) , \varphi \circ \sigma (t_{2,\pi} ) \in \Fil ^{0} B_\dR-\Fil ^{1} B_\dR$.
\item  $ \varphi^2(t_{2,\pi} ) = \pi \cdot t_{2,\pi} $.
\item The elements $t_{2,\pi}  ,  \sigma(t_{2,\pi} ), \varphi (t_{2,\pi} )$ and $ \varphi \circ \sigma (t_{2,\pi} )  $ are invertible in $B_{\crys,F}$.
\end{enumerate}
\end{prop}
\begin{proof}
This period will appear as the period of an explicit filtered $\varphi$-module. Let $x^2+ax+b$ be the minimal polynomial of $1/\pi$ over $\Q_p$.  Define the $2 \times 2$ matrix
\[ 
C =  \begin{pmatrix} 0 & -b \\
1 & -a
 \end{pmatrix}.
 \]
 Consider the vector space  $N=\Q_p^4$  endowed with the Frobenius given by the block matrix
 \[ 
\varphi =  \begin{pmatrix} 0_2 & C \\
I_2 & 0_2
 \end{pmatrix},
 \]
where $I_2$ and $0_2$ are the identity and the zero $2 \times 2$  matrices.  

 Notice  that we have \[ 
\varphi^2 =  \begin{pmatrix} C & 0_2 \\
0_2 & C
 \end{pmatrix}.
 \]
 Hence the action of $\varphi^2$ induces a decomposition of eigenspaces
  \[N \otimes F=N_{1/\pi}\otimes N_{1/\sigma(\pi)}.\]
 Fix a nonzero eigenvector $m \in N_{1/\pi}$ hence satisfying
 \begin{equation}\label{troilo}  
\varphi^2 (m)=  \frac{1}{\pi} m.
 \end{equation}
Define the filtration to be \[\Fil^{-1} (N\otimes F)=N\otimes F, \, \Fil^0 (N\otimes F)= N_{1/\sigma(\pi)}\oplus F\cdot \varphi (m ) , \, \Fil^{1} (N\otimes F)=0.\] 
We claim that this filtered $\varphi$-module is admissible. First notice that there are no non trivial submodules of $N$ stable by $\varphi$. Indeed the action of $\varphi$ on $N$ satisfies the equation $x^4+ax^2+b=0$. This polynomial is irreducible over $\Q_p$ as it is also the minimal polynomial of $1/\sqrt{\pi}$, which has valuation $1/4$.
 We deduce that the admissibility condition has to be checked only on $N$. Moreover the action of $\varphi$ on $\det N$, the maximal exterior power of $N$, is by multiplication by $b$, which has valuation equal to $-1$ and the only non trivial step of the induced filtration on $\det N$ is at $-1$ as well.
 
As $N$ is an admissible filtered $\varphi$-module we can apply Theorem \ref{tsuji} and consider the $\Gal_{\Q_p}$-representation $V$ which verifies  $N=D(V)$. 
Now notice that the action of $\varphi^2$ on $N$ respects the filtration, this means that the field $F$ acts on $N$ as filtered $\varphi$-module. 
Then Theorem \ref{tsuji}  implies that
 $F$ acts on $V$ as well and that the canonical isomorphism  $V \otimes B_\crys= N \otimes B_\crys$ is $F$-equivariant. In particular there is a decomposition in eigenspaces  
 \[V \otimes F=V_{1/\pi}\otimes V_{1/\sigma(\pi)}\]
 and, if we fix  a nonzero eigenvector $v \in V_{1/\pi}$, there exist two periods $\alpha, \beta$ in $B_{\crys,F}$ such that 
  \begin{equation}\label{troilo2}  
\alpha \cdot m+ \beta \cdot \varphi(m)=   v.
 \end{equation}
 
As the Frobenius acts trivially on $V$, relations (\ref{troilo}) and (\ref{troilo2}) imply $ \varphi(\alpha) = \beta$ and $ \varphi^2(\alpha) = \pi \alpha$ . We claim now that the period $t_{2,\pi}=\alpha$ satisfies all the properties of the statement. 

 By Theorem \ref{tsuji} we have $v\in \Fil^0[N_{1/\pi} \otimes B_\dR]$. This means that the vector $v$ can be written as $\gamma n + \delta \varphi(m)$ with $\gamma \in \Fil^{1}B_\dR, \delta \in \Fil^0B_\dR$ and $n \in N_{1/\pi}$, which implies that $\alpha \in \Fil^{1}B_\dR$ and  $\beta \in \Fil^{0}B_\dR$. 

Now notice that, for all $x\in \Q_{p^2}$, the vector $x\alpha \cdot m+ \varphi(x)\varphi(\alpha) \cdot \varphi(m)$ is in the vector space $V=[N\otimes B_\crys]^{\varphi=\id} \cap  \Fil^0[N \otimes B_\dR] $. 
Moreover, it is in the eigenspace $V_{1/\pi}$. 
Hence,  the one dimensional $F$-vector space of $p$-adic periods relating $\det N_{1/\pi}$ to $\det V_{1/\pi}$ is generated by an element in $B_{\crys,F}$ of the form $\alpha \varphi(\alpha) (x-\varphi( x) )$, for any fixed $x \in \Q_{p^2} - \Q_p$. 
As $\det N_{1/\pi}\otimes B_{\crys,F}$ and $\det V_{1/\pi}\otimes B_{\crys,F}$ are free $B_{\crys,F}$-modules of rank one the element $\alpha \varphi(\alpha) (x-\varphi( x) )$ must be invertible in $B_{\crys,F}$, hence $\alpha$ and  $\varphi(\alpha)$ are invertible as well.
 
 Finally, let us  consider the valuation $\nu$ from Theorem \ref{javier} associated  to the filtration. 
 First of all, as the only non trivial step of the filtration on $\det N_{1/\pi}$ is  $-1$, we deduce that $\nu ( \alpha \varphi(\alpha) (f-\sigma( f) ))=1$ and hence $\nu ( \alpha \varphi(\alpha))=1$. 
 On the other hand we have showed the inequalities $\nu ( \alpha) \geq 1$ and  $\nu (  \varphi(\alpha))\geq 0$. 
 Combining them with $\nu ( \alpha \varphi(\alpha)  )=1$ we deduce that these inequalities are actually equalities.
 
 This shows all the desired properties on $\alpha$ and $ \varphi(\alpha)$. 
 To show the analogous properties on $\sigma (\alpha)$ and $ \varphi \circ \sigma (\alpha)$ one applies $\sigma$ on the relation (\ref{troilo2}) and argues as above replacing the eigenspaces $N_{1/\pi}$ and $V_{1/\pi}$ with the eigenspaces $N_{1/\sigma(\pi)}$  and $V_{1/\sigma(\pi)}$.
\end{proof}

\begin{prop}\label{lubintatewild2} (Colmez)
Consider the ring $B_{\crys,F}$ for the prime $p=2$ and $F=\Q_2(\sqrt{3}).$
There  exists an element $t_{f} \in B_{\crys,F}$ which verifies the following properties.
\begin{enumerate}
\item $t_{f} \in \Fil ^{1} B_\dR-\Fil ^{2} B_\dR$.
\item $\sigma (t_{f}) , \varphi (t_{f}) , \varphi \circ \sigma (t_{f}) \in \Fil ^{0} B_\dR-\Fil ^{1} B_\dR$.
\item  $ \varphi^2(t_{f}) = (1-\sqrt{-1})  \cdot t_{f}$.
\item The elements $t_{f}  ,  \sigma(t_{f} ), \varphi (t_{f} )$ and $ \varphi \circ \sigma (t_{f} )  $ are invertible in $B_{\crys,F}$.
\end{enumerate}
\end{prop}
\begin{proof}
This is a special case of Proposition \ref{idearef} up to a little subtlety. Consider first the ramified extension $F=\mathbb{Q}_2(\sqrt{-1})$ and the uniformizer $\pi= 1-\sqrt{-1}. $ Then Proposition  \ref{idearef}  will give a period $t_{2,\pi}$ which verifies all the properties of the statement except that it is constructed in $B_{\crys,\mathbb{Q}_2(\sqrt{-1})}$. On the other hand we have the equality \[B_{\crys,\mathbb{Q}_2(\sqrt{-1})} = B_{\crys,\mathbb{Q}_2(\sqrt{3})}.\]
Notice though that the two descriptions of this ring exchange the role of $ \sigma $ with $ \varphi \circ \sigma$. As  Proposition \ref{idearef}  is stable under this exchange, the period $t_f=t_{2,\pi}$ does satisfy all the desired properties.
\end{proof}
\begin{corollary}\label{brinonwild2}
  Consider $F=\Q_2(\sqrt{3})$ and let $\alpha \in \Q_{2^4}(\sqrt{3})=\Q_{2^4}(\sqrt{-1})$ be an element  such that
\[\varphi(\alpha) =  \sqrt{-1} \cdot \alpha.\]
Define $\lambda_i \in B_{\crys,F}$ as
\[\lambda_i = \left(\frac{\alpha \cdot t_{f} \cdot \varphi(t_{f}) }{\sigma(t_{f})   \cdot (\varphi \circ \sigma) (t_{f})}\right)^i .\]
Then $\lambda_i$ is a nonzero element of  $P_{i}(F)$ and  moreover it verifies
\[\lambda_i \cdot\sigma(\lambda_i) = (\alpha  \cdot \sigma(\alpha))^i.\] 
\end{corollary}
\begin{proof}
All properties follow from Proposition \ref{lubintatewild2}  by direct calculation.
\end{proof}

\section{End of the proof}\label{endpf}
We are now ready to show Theorem \ref{mainthm}  (via Proposition \ref{padicreduction}).  

\begin{proof}[Proof of Theorem \ref{mainthm}]
Thank to Proposition \ref{ordinarycase}, we can work under Assumption \ref{ipos}, and fix the positive integer $i$ as in the assumption as well as the field $F$ from Proposition \ref{propendo}.

Let us first suppose that  $F$ is unramified. We can then combine Proposition \ref{warmbrinon} and Corollary \ref{brinon} to conclude that the quadratic forms  $q_{B,p}$ and $q_{Z,p}$ are isomorphic if and only if $1/p^i$ is a norm. Because of Lemma \ref{notnorm}, this is the case   if and only if $i$ is even.

Suppose now that $F$ is ramified but it is not $\mathbb{Q}_2(\sqrt{3})$ nor $\mathbb{Q}_2(\sqrt{-1})$. Then, following  Remark \ref{remtame}, we can combine Proposition \ref{warmbrinonram} and Corollary \ref{brinonram} to conclude that the quadratic forms $q_{B,p}$ and $q_{Z,p}$ are isomorphic if and only if    $ a^i$ is a norm. Because of Lemma \ref{notnormram}, this is the case   if and only if $i$ is even.

If $F$ is  $\mathbb{Q}_2(\sqrt{-1})$ we can  combine Proposition \ref{warmbrinonram} and Corollary \ref{brinonwild} to conclude that the quadratic forms $q_{B,p}$ and $q_{Z,p}$ are isomorphic if and only if    $  (\alpha  \cdot \sigma(\alpha) )^i$ is a norm. Because of Lemma \ref{notnormwild}, this is the case   if and only if $i$ is even.

Finally, if $F$ is  $\mathbb{Q}_2(\sqrt{3})$ we can  combine Proposition \ref{warmbrinonram} and Corollary \ref{brinonwild2} to conclude that the quadratic forms  $q_{B,p}$ and $q_{Z,p}$ are isomorphic if and only if    $  (\alpha  \cdot \sigma(\alpha) )^i$ is a norm. Because of Lemma \ref{notnormwild2}, this is the case   if and only if $i$ is even.

In conclusion, we showed for that, for any possible $F$, the quadratic forms $q_{B,p}$ and $q_{Z,p}$ are isomorphic if and only if $i$ is even. This shows Theorem \ref{mainthm}  via Proposition \ref{padicreduction}.
\end{proof}

\begin{lemma}\label{notnorm}
The element $1/p\in \Q_p^*$ does not belong to the group of norms $ N_{\Q_{p^2}/\Q_p}(\Q^*_{p^2}).$
\end{lemma}
\begin{proof}
For $f\in \Q_{p^2}$,   the $p$-adic valuation of $f$ and $\varphi (f)$ coincide, hence a norm must have even $p$-adic valuation.
\end{proof}

\begin{lemma}\label{notnormram}
Suppose that $F$ is ramified but it is not $\mathbb{Q}_2(\sqrt{3})$ nor  $\mathbb{Q}_2(\sqrt{-1})$. Let $\sqrt{a}\in \Q_{p^2}$ be as in Corollary \ref{brinonram}. Then the element $a \in \Q_p^*$ does not belong to the group of norms $ N_{F/\Q_p}(F^*).$
\end{lemma}
\begin{proof}
This is about finding nonzero solutions $ x,y \in \Q_p$ of the equation
\begin{equation}\label{calcare} x^2 -\pi^2 y ^2 = a. \end{equation}
We can suppose that    $a$ has $p$-adic valuation zero. Then $x$ has valuation zero and $y$ has non-negative valuation. 

Consider first the case $p\neq 2$. By reducing modulo $p$ we see that the existence of a solution $(x,y)$ would imply that $a$ would be a square in $\mathbb{F}_p$  which gives a contradiction. 

Consider now the case $p= 2$. We can then take $a=5$ and $\pi^2=2,6,10$ or $14$, see Remark \ref{remtame}. If we reduce modulo $8$ the  equation (\ref{calcare}) we see that the existence of a solution $(x,y)$ would imply that  $3,5$ or $7$ would be a square in $\mathbb{Z}/ 8 \mathbb{Z}$ which gives a contradiction. 
\end{proof}

\begin{lemma}\label{notnormwild}
Let $\alpha \in \Q_{2^4}(\sqrt{3})=\Q_{2^4}(\sqrt{-1})$ be as in Corollary \ref{brinonwild}. Then the element $\alpha  \cdot \sigma(\alpha) \in \Q_2^*$ does not belong to the group of norms $ N_{\mathbb{Q}_2(\sqrt{-1})/\Q_2}(\mathbb{Q}_2(\sqrt{-1})^*).$
\end{lemma}
\begin{proof}
 Write $\alpha= \alpha_1+\sqrt{-1}\cdot\alpha_2$ with $\alpha_1,\alpha_2 \in \Q_{2^4}$, then we have
 \[\varphi(\alpha_1)=-\alpha_2,  \hspace{1cm} \varphi(\alpha_2)=\alpha_1, \hspace{1cm} \alpha  \cdot \sigma(\alpha)=\alpha_1^2+\varphi(\alpha_1^2).\]
 Let us make explicit computations. Fix the presentations
\[\mathbb{F}_{2^2}=\mathbb{F}_{2}[x]/(x^2+x+1) \hspace{1cm} \mathbb{F}_{2^4}=\mathbb{F}_{2^2}[y]/(y^2+xy+1)\]
and write
\[\mathbb{Q}_{2^2}=\mathbb{Q}_{2}(\frac{-1+\sqrt{-3}}{2}) \hspace{1cm} \mathbb{Q}_{2^4}=\mathbb{Q}_{2^2}[y]/(y^2+(\frac{-1+\sqrt{-3}}{2})y+1).\]
Then $\Delta=\sqrt{(\frac{-1+\sqrt{-3}}{2})^2-4}=\sqrt{\frac{-1-\sqrt{-3}}{2}-4}$ can be chosen as $\alpha_1$. Thus, 
\[\alpha  \cdot \sigma(\alpha)=\Delta^2 + \varphi(\Delta^2)= -9.\]
This is a norm if and only if one can find a nonzero solution $u,v \in \Q_2$ of the equation $u^2+v^2=-9$. This is impossible by looking at the equation modulo $8$.
\end{proof}

\begin{lemma}\label{notnormwild2}
Let $\alpha \in \Q_{2^4}(\sqrt{3})=\Q_{2^4}(\sqrt{-1})$ be as in Corollary \ref{brinonwild2}. Then the element $\alpha  \cdot \sigma(\alpha) \in \Q_2^*$ does not belong to the group of norms $ N_{\mathbb{Q}_2(\sqrt{3})/\Q_2}(\mathbb{Q}_2(\sqrt{3})^*).$
\end{lemma}
\begin{proof}
 We argue as in the proof of Lemma \ref{notnormwild}. Write $\alpha= \alpha_1+\sqrt{-1}\cdot \alpha_2$ with $\alpha_1,\alpha_2 \in \Q_{2^4}$, then we have
\[\varphi(\alpha_1)=-\alpha_2,  \hspace{1cm} \varphi(\alpha_2)=\alpha_1, \hspace{1cm}\]

 Moreover, as the elements fixed by $\sigma$ are exactly those in $B_{\crys}$, we have
 \[ \sigma(\sqrt{-1})=-\sqrt{-1} \hspace{1cm} \sigma(\alpha_1)= \alpha_1 \hspace{1cm} \sigma(\alpha_1)= \alpha_1 \]
 and we deduce that $\alpha  \cdot \sigma(\alpha)=\alpha_1^2+\varphi(\alpha_1^2).$ This means that we can use the same computations as in  the proof of Lemma \ref{notnormwild} and for the same choice of $\alpha$ in there we have \[\alpha  \cdot \sigma(\alpha) = -9.\]
This is a norm if and only if one can find a nonzero solution $u,v \in \Q_2$ of the equation $u^2 - 3 v^2=-9$. This is impossible by looking at the equation modulo $8$.
\end{proof}

\begin{remark}
Let us point out that the formulation of Proposition \ref{padicreduction} does not involve the ring $B_{\crys}$. We do not know if there is a way to show Theorem \ref{mainthm} via   this proposition without computing explicitly the $p$-adic periods.
\end{remark}

\begin{appendices}

\section{Geometric examples}\label{geometricexamples}
In this section we discuss several examples to which Theorems \ref{thmscht} and \ref{mainthm} apply non-trivially. We are  particularly interested in exotic classes on abelian fourfolds (Definition \ref{definizione}). The main result is Proposition \ref{propexemple} where we discuss the existence of exotic classes that cannot be lifted to algebraic classes in characteristic zero. The techniques of construction there are inspired by \cite{LeOo} and
 \cite{Zah}. 
 
  Other examples of exotic classes will be found  in Remark \ref{altreesotiche}. We end the section with an example (other than abelian fourfolds) for which the standard conjecture of Hodge type holds true via Theorem \ref{mainthm}.

\begin{proposition}\label{propexemple}
Let $p$ be a prime number and  let $K$ be an  imaginary quadratic number field where $p$ does not split. Then there exists an abelian fourfold $A$ over $\overline{\mathbb{F}}_p$ verifying the following properties.
\begin{enumerate}
\item The endomorphism  algebra is a number field that can be written as the compositum \[\End (A)_\Q = K \cdot R\] where $R$ is a totally real number field such that $[R:\Q]=4$.
(In particular $A$ is simple and has a unique CM-structure).
\item Consider the motivic decomposition from Proposition \ref{primadecomposizione}(\ref{decomposizioneQ}). Among the factors $\mathcal{M}_I$ of $\mathfrak{h}^4(A)$ there exists a (unique) factor $M$  such that $M(2)$ does not contain any Lefschetz class (Definition \ref{defexotic}) but it is Frobenius invariant (for a model of $A$ defined over a finite field).
\item\label{cond2} For any CM-lifting  of $A$ to characteristic zero (see Theorem \ref{HondaTate} and Corollary \ref{lifting motives}) the Hodge structure $R_B(M)$ is of type $(3,1),(1,3)$.
\item\label{cond1} For any CM-lifting  of $A$ to characteristic zero we have the equality $\End_{\NUM(\C)_\Q}  (M) = K.$ 
\end{enumerate}
\end{proposition}
\begin{proof}
The proof is decomposed in a series of lemmas. The final step is Lemma \ref{lemmafine}.
\end{proof}
\begin{remark}\label{bellaciao} Let us make some comments on the above proposition.
\begin{enumerate}
\item If (a model of) the fourfold $A$ verifies the Tate conjecture, then $M$ is an exotic motive in the sense of Section \ref{sectionexotic}. In particular, in each characteristic, there should be infinitely  many non-isogenous abelian fourfolds having exotic motives. 

If the Tate conjecture was highly false and no such exotic classes existed then Theorem \ref{thmscht} would follow directly from the arguments in Section \ref{sectionlefschetz}.

The Tate conjecture and the standard conjecture of Hodge type   should be thought as two independent and different  problems. The first is about the construction of algebraic classes, the second is about how they intersect (independently whether there are a lot of algebraic classes or not). It seems likely that a solution of one problem does not imply a solution for the other. For example the proof of the Tate conjecture for divisors on abelian variety  \cite[Theorem 4]{Tate} does not imply the standard conjecture of Hodge type for divisor on abelian variety (which is known by a different argument).

\item Because of the Hodge types in part (\ref{cond2}), the (expected) algebraic classes in positive characteristic cannot be lifted to algebraic classes in characteristic zero.

\item\label{unamattina} Notice that the field $F=K \otimes_\Q \Q_p$ can be any quadratic extension of $\Q_p$. This field $F$ coincides with the one in Proposition \ref{propendo}. This shows that the different cases studied in Sections \ref{p-adic periods} and \ref{computation} were needed.
\item The hypothesis that $K$ is totally imaginary is necessary. If $K$ were a real quadratic  number field,  conditions    (\ref{cond2}) and (\ref{cond1}) in Proposition  \ref{propexemple} would not be compatible (see the proof of  Lemma \ref{bellaidea}).
\end{enumerate}
\end{remark}

\begin{lemma}\label{marx}
Let $p$ and $K$ be as in Proposition \ref{propexemple}. There exists a totally real number field $R$  such that the following holds:
\begin{enumerate}
\item  The prime $p$ does not split in $R$.
\item The degree $[R:\Q]$ is four.
\item The field $K \otimes_\Q \Q_p$ is embeddable in the field $R \otimes_\Q \Q_p$. 
\item If $\tilde{R}$ is the normal closure of $R$ over $\Q$, then $\Gal(\tilde{R}/\Q)=\mathfrak{S}_4$. In particular we have $\Gal(\tilde{R}/R)=\mathfrak{S}_3$.
\end{enumerate}
\end{lemma}
\begin{proof}
Same as in \cite[\S 3]{LeOo}.
\end{proof}

\begin{lemma}\label{lenin}
Let $R$ be as in the above lemma. The following holds:
\begin{enumerate}
\item  The subfields of the compositum $R\cdot K$ are $\Q,K,R$ and $R\cdot K$.
\item The prime $p$ factorises in $R\cdot K$ as 
\[p=(\mathfrak{p}\cdot \bar{\mathfrak{p}})^e,\]
where $\mathfrak{p}$ and $\bar{\mathfrak{p}}$ are  two prime ideals which are exchanged by  complex conjugation and $e$ is  the ramification index.
\end{enumerate}
\end{lemma}
\begin{proof}
Let $\tilde{R}$ be as in Lemma \ref{marx}. The equality $\Gal(\tilde{R}/\Q)=\mathfrak{S}_4$ implies
\[\Gal(\tilde{R}\cdot K/\Q)=\mathfrak{S}_4\times \Z/2\Z\]
and similarly $\Gal(\tilde{R}/R)=\mathfrak{S}_3$ implies $\Gal(\tilde{R}\cdot K/R \cdot K)=\mathfrak{S}_3.$
We deduce that the subfields of $R\cdot K$ are in bijection with the subgroups of $\mathfrak{S}_4\times \Z/2\Z$ containing $\mathfrak{S}_3$. Those are precisely $\mathfrak{S}_3 , \mathfrak{S}_4, \mathfrak{S}_3\times \Z/2\Z$ and $\mathfrak{S}_4\times \Z/2\Z$, which implies (1).

Note that  the equality $R\cdot K \cong R \otimes_\Q K$ holds. Hence we have
\[(R\cdot K) \otimes_\Q \Q_p \cong  (K \otimes_\Q \Q_p) \otimes_{\Q_p} (R \otimes_\Q \Q_p) \cong (R \otimes_\Q \Q_p)^{\Gal(K \otimes_\Q \Q_p/\Q_p) }, \]
where the last equality comes from  Lemma \ref{marx}(3). As ${R \otimes_\Q \Q_p}$ is a field (by Lemma \ref{marx}(1)), the prime $p$ factorizes in $R\cdot K$ as the product of two different primes. Moreover, those two primes are exchanged by $\Gal(K \otimes_\Q \Q_p/\Q_p)$. As complex conjugation generates this Galois group it exchanges these two prime ideals. In particular they must have the same ramification index. This concludes part (2).
\end{proof}

\begin{lemma}\label{engels}
Consider the prime ideals $\mathfrak{p}$ and $\bar{\mathfrak{p}}$ of  $R\cdot K$ from the above lemma. Then there exist an integer $n$, a $p$-power $q$ and a $q$-Weil number $\alpha\in R\cdot K$ verifying the following properties:
\begin{enumerate}
 \item The ideal generated by $\alpha$ factorizes as
\[(\alpha)=\mathfrak{p}^n\cdot \bar{\mathfrak{p}}^{3n}.\]
\item For any positive integer $s$, we have the equality
\[\Q(\alpha^s)= R\cdot K.\]
\item The norm $N_{R\cdot K/K}(\alpha)\in K$ equals $q^2$.
\end{enumerate}
\end{lemma}
\begin{proof}
Consider the ideal $I=(\mathfrak{p}\cdot \bar{\mathfrak{p}}^{3})^e $. We have that $I \cdot \bar{I}= p^4$ by Lemma \ref{lenin}(2). Actually, for any $g\in\Gal(\tilde{R}\cdot K/\Q),$ we have the equality ${g(I) \cdot \overline{g(I)}= p^4}$ of ideals in $\tilde{R}\cdot K$. This follows from the case $g=\id$ together with the explicit description of the Galois group in the proof of the lemma above (which implies that complex conjugation is in the center of the Galois group). Hence, we can apply  \cite[Lemma 1]{Honda} and deduce that there exists a  $q$-Weil number $\alpha\in R\cdot K$ verifying (1).

Let us show (2). Because of Lemma \ref{lenin}(1), this amounts to showing that $\alpha^s$ does not belong to $K$ nor to $R$. Now, if  $\alpha^s$ belongs to $K$ (or to $R$) then its factorization in $R\cdot K$ would have the same exponent in $\mathfrak{p}$ and in $\bar{\mathfrak{p}}$ because there is only one prime above $p$ in $K$ (or in $R$); see the hypothesis on $K$ in Proposition \ref{propexemple} (respectively by Lemma \ref{marx}(1)).

Let us now show (3). If we compute the  norm of the ideal $(\alpha)$ we obtain
\[ (N_{R\cdot K/K} (\alpha))= N_{R\cdot K/K} (\mathfrak{p}^n\cdot \bar{\mathfrak{p}}^{3n}) = N_{R\cdot K/K} (\mathfrak{p})^n\cdot N_{R\cdot K/K} ( \bar{\mathfrak{p}})^{3n}= \tilde{p}^m,\]
where $m$ is an integer and $\tilde{p}$ is the only prime ideal above $p$ in $K$. Hence, after possibly replacing $\alpha$ by a power, we obtain that the ideal $N_{R\cdot K/K} (\alpha)$ is generated by a power of $p$. For weight reasons we have
\[(N_{R\cdot K/K} (\alpha)) = (q^2).\]
This is equivalent to the relation $N_{R\cdot K/K} (\alpha) = \xi \cdot q^2$, where $\xi$ is an invertible element of the ring of integers of $K$. As the group of  invertible elements of the ring of integers of an  imaginary quadratic field is finite, after replacing $\alpha$ by a power we get (3). Notice that such a power of $\alpha$ will still have properties (1) and (2).
\end{proof}

\begin{lemma}\label{elche}
Let $\alpha$ be a $q$-Weil number verifying the properties as in the above lemma. Let $A$ be an abelian variety over $\mathbb{F}_{q}$ whose isogeny class  corresponds to $\alpha$ under the Honda--Tate correspondance \cite{TateBour}. Then the following holds:
\begin{enumerate}
\item  The dimension of $A$ is four.
\item  $A$ is geometrically simple.
\item $\End (A)_\Q = \End (A_{\bar{\mathbb{F}}_p})_\Q = K \cdot R.$
\item The slopes of $A$ are $(1/4,3/4)$
\item There are Frobenius invariant classes in $H_\ell^4(A)$ which are not of Lefschetz type.
\end{enumerate}
\end{lemma}
\begin{proof}
By \cite{Tate} the division algebra $ \End (A)_\Q$ has center equal to $\Q(\alpha)$ which is  $R\cdot K$ by Lemma \ref{engels}. Moreover, by \cite[Theorem 1]{TateBour}, this division algebra  splits at every place except possibly at  the places $\mathfrak{p}$ and $\bar{\mathfrak{p}}$ above $p$. The local invariants there are computed by the formula in \cite[Theorem 1]{TateBour}, which gives
\[\inv_\mathfrak{p}( \End (A)_\Q)=\frac{v_\mathfrak{p}(\alpha)}{v_\mathfrak{p}(q)}\cdot [(R\cdot K)_\mathfrak{p}:\Q_p] \hspace{0.5cm} \mod \Z\]
and similarly for $\bar{\mathfrak{p}}$. 

We claim that these local invariants are trivial as well. Indeed, using the factorisation in Lemma \ref{engels}(1), we deduce that $\frac{v_\mathfrak{p}(\alpha)}{v_\mathfrak{p}(q)}=\frac{1}{4}.$ On the other hand, the degree $[(R\cdot K)_\mathfrak{p}:\Q_p]$ equals four, because  $[R\cdot K:\Q]=8$ and the two primes above $p$ are exchanged by  complex conjugation (Lemma \ref{lenin}(2)). Altogether we have that $\inv_\mathfrak{p}( \End (A)_\Q)=0$ and similarly one shows $\inv_{\bar{\mathfrak{p}}}( \End (A)_\Q)=0$ 

Because all the invariants of the $(R\cdot K)$-central algebra $ \End (A)_\Q$ are trivial, we have 
$R\cdot K= \End (A)_\Q.$  As ${[R\cdot K:\Q]=8}$  we deduce (1). 

Consider now the abelian variety $A_s$  over $\mathbb{F}_{q^s}$ whose isogeny class  corresponds to $\alpha^s$. Following the Honda--Tate correspondance $A_s$ is a simple factor of $A \times_{\mathbb{F}_q} \mathbb{F}_{q^s} $. On the other hand, all the arguments above work  by replacing $\alpha$ by  $\alpha^s$, because of  Lemma \ref{engels}(2). In particular, $A_s$  has also dimension four and $R\cdot K= \End (A_s)_\Q.$ This implies (2) and (3).

One slope has already been computed, namely $\frac{v_P(\alpha)}{v_P(q)}=\frac{1}{4}.$ Duality implies that there is also the slope $3/4$. As $A$ has dimension four there are no more slopes.

Let us  now show (5). The existence of a class such as the ones claimed is equivalent to the existence of a set $I$ consisting of four Galois conjugates of $\alpha$  whose product equals $q^2$  and such that $I$ is not stable under the action of   complex conjugation (see the proof of  Lemma \ref{musica} or \cite[\S 2]{Zah}). We claim that the relation
\[N_{R\cdot K/K}(\alpha) = q^2 \]
(Lemma \ref{engels}(3))  gives precisely the existence of those four Galois conjugates. Indeed, let $\tilde{R}$ be the normal closure of $R$ over $\Q$, by definition we have
\[N_{R\cdot K/K}(\alpha) = \prod_{g \in \Hom_K( R\cdot K , \tilde{R}\cdot K)} g(\alpha).\]
Hence it is enough to show that  complex conjugation does not stabilize the set $J=\{g(\alpha)\}_{g \in \Hom_K( R\cdot K , \tilde{R}\cdot K)}.$ As the set   $J$ is of size four and the total Galois orbit of $\alpha$ is of size $8$ there is an element of $\Gal(\tilde{R}\cdot K/\Q)$ which does not stabilize $J$.
On the other hand, thanks to the equality \[\Gal(\tilde{R}\cdot K/\Q) = \Gal(\tilde{R}\cdot K/K) \times \Gal(\tilde{R}\cdot K/\tilde{R})\]
we have that the total Galois group $\Gal(\tilde{R}\cdot K/\Q)$ is generated by its subgroup $ \Gal(\tilde{R}\cdot K/K) $ and  complex conjugation. As  $ \Gal(\tilde{R}\cdot K/K) $ stabilizes $J$,   complex conjugation cannot stabilize it.
\end{proof}
\begin{lemma}\label{lemmafine}
Let $A$ be an abelian fourfold which satisfies the properties of the lemma above. Then it also satisfies all the conditions of Proposition \ref{propexemple}.
\end{lemma}
\begin{proof}
Part (1) has already been showed. Part (2) follows from Lemma \ref{elche}(5). (Unicity comes from Lemma \ref{lemmafour}.)

Let us now show part (3). Write $\alpha, \beta, \gamma, \delta, q/\alpha, q/\beta, q/\gamma, q/ \delta $ for the eight (distinct) Frobenius eigenvalues and consider the decomposition in eigenlines
\[\mathfrak{h}^1(A)=V_\alpha \oplus V_\beta \oplus V_\gamma \oplus V_\delta \oplus V_{q/\alpha} \oplus V_{q/\beta} \oplus V_{q/\gamma} \oplus V_{q/ \delta} \]
as in Proposition \ref{decomposition}. Among these eight eigenvalues, four have slope $1/4$ and four have slope $3/4$ and they are exchanged by complex conjugation.

Fix a CM-lifting (Theorem \ref{HondaTate}). The above decomposition in eigenlines will lift as well (Corollary \ref{lifting motives}). Among the eight lines, four will belong to $H^{1,0}$ and four will belong to $H^{0,1}$ and again they are exchanged by complex conjugation. The Shimura--Taniyama  formula  \cite[Lemma 5]{TateBour}  implies that there is exactly one eigenvalue, call it $\alpha$, whose slope is $1/4$ and such that
$V_\alpha \subset H^{1,0}$. Equivalently,  there is exactly one eigenvalue, namely $q/\alpha$, whose slope is $3/4$ and such that
$V_{q/\alpha} \subset H^{0,1}$. 

Now decompose $M= M_{\alpha, \beta, \gamma, \delta} \oplus M_{q/\alpha, q/\beta, q/\gamma, q/ \delta}$ via Proposition \ref{primadecomposizione}. (After possibly renaming the eigenvalues.)
With this notation we have the relation 
\[\alpha\cdot \beta \cdot \gamma \cdot \delta=q^2.\]
By looking at the $p$-adic valuation we deduce that, among $\beta, \gamma, \delta$ there is exactly one eigenvalue of slope $1/4$, say $\beta$.  Hence  we have $V_\beta \subset H^{0,1}$ and  $V_\gamma, V_\delta \subset H^{1,0}$. Altogether we deduce
\[V_\alpha \otimes V_\beta \otimes V_\gamma \otimes V_\delta \subset H^{3,1}\]
which gives (3).

\

Let us now show part (4). By construction we can find a quadratic number field $F\subset \tilde{R}\cdot K$   such that the motive $M$  decomposes in the category $\CHM (k)_{F}$ into a sum 
\[\mathcal{M}_{I} = M_{I} \oplus  M_{\bar{I}}  \] of two motives of rank one (see Proposition \ref{primadecomposizione}). We first claim that such a field $F$ must be imaginary. If $F$ were contained in $\R$ then the Betti realization of the lifting of $M_{I}$ would respect the Hodge symmetry. As it is one dimensional for weight reasons it would be of type $(2,2)$. This contradicts part (3).

By \cite{Jann},  $D=\End_{\NUM(\C)_\Q}  (M)$  is a division algebra. By construction, $F$ splits $D$. We claim that $D=F$. Otherwise we would have $D \otimes F \cong M_{2 \times 2}(F)$ which would imply that $M_{I}$ and $M_{\bar{I}} $ are isomorphic as numerical motives. As homological and numerical equivalence is known to coincide for complex abelian varieties \cite{Lieb}, this would imply that their Betti realization are isomorphic, which is impossible because of the different Hodge types.

In conclusion,  $\End_{\NUM(\C)_\Q}  (M)$ is an imaginary quadratic field contained in $ \tilde{R}\cdot K$. On the other hand,  there is only one such field (namely $K$) because of the description of  $\Gal (\tilde{R}\cdot K/\Q)$ in the proof of Lemma \ref{lenin}.
\end{proof}
\begin{remark}\label{altreesotiche}
Let us comment on other examples of exotic motives coming from abelian fourfolds.
\begin{enumerate}
\item One can construct  an abelian fourfold $A$ over a finite field having an exotic motive whose lifting to $\C$ has Betti realization of type $(2,2)$. Such a condition means that  the CM-lifting of $A$  over $\C$   has Hodge classes which are not Lefschetz. This   situation (over $\C$) has been classified in \cite{MooZa}. So, any reduction modulo $p$ of their examples will give an abelian fourfold over a finite field of the desired type. (To avoid that the reduction modulo $p$ creates more Lefschetz classes, one can take an ordinary prime). 

As already pointed out, these examples are less interesting for the standard conjecture of Hodge type, see Remark \ref{boris}(\ref{duccio}).
\item\label{altreesotiche2} There are no exotic motives over $\overline{\mathbb{F}}_p$ (coming from abelian fourfolds) whose lifting  to $\C$ have Betti realization    of type $(4,0),(0,4)$. To show this, consider a model of the abelian fourfold over a finite field $\mathbb{F}_{q}$. Let $I$ be the set of Frobenius eigenvalues such that the corresponding eigenspaces are lifted into $H^{1,0}$ and $\bar{I}$ be the set of Frobenius eigenvalues such that the corresponding eigenspaces are lifted into $H^{0,1}$. If the cohomology group $H^{4,0}\oplus H^{0,4}$ becomes Frobenius invariant over $\mathbb{F}_q$, then $\prod_{\alpha\in I} \alpha = \prod_{\beta\in \bar{I}} \beta(=q^2).$ On the other hand, using the Shimura--Taniyama formula  \cite[Lemma 5]{TateBour}, we have that the $p$-adic valuation of  $\prod_{\alpha\in I} \alpha$ is greater than the one of $ \prod_{\beta\in \bar{I}} \beta$, except if all Frobenius eigenvalues have the same slopes. In this case the abelian variety would be isogenous to the forth power of a supersingular elliptic curve and hence all algebraic classes would be Lefschetz.
\item Having the results of Section 7 in mind, the last example  that needs to be discussed is that of an abelian fourfold with a four dimensional space of exotic classes. By Lemma \ref{lemmafour}, this reduces to an abelian fourfold over  $\mathbb{F}_{q^2}$  of the form $X\times E$, where $X$ is an abelian threefold and $E$ is a supersingular elliptic curve on which   Frobenius acts as $q \cdot \id$. Now the equations  (\ref{eq:4}) and  (\ref{eq:5}) imply that the existence of an exotic class on  $X\times E$ is equivalent to the existence of an exotic class on $X^2$. There are infinitely many such threefolds $X$, they have been classified in \cite{Zah}. 
\end{enumerate}
\end{remark}
\begin{proposition}\label{propfermat}
The standard conjecture of Hodge type holds true for Fermat's cubic fourfold
$X=\{x_0^3+\cdots +x_5^3=0\}\subset \mathbb{P}_k^5$ over any field $k$ (of characteristic different from $3$).
\end{proposition}
\begin{proof}
Let us first consider $X$ as variety over $\C$.  By \cite[Proposition 11]{Beauville} its Hodge structure decomposes as
\[H_B^*(X,\Q)=\Q(0) \oplus \Q(-1) \oplus \Q(-2)^{\oplus 21} \oplus V_B \oplus \Q(-3) \oplus \Q(-4)\]
where   $V_B$ is a $\Q$-Hodge structure of rank $2$ and of type $(3,1), (1,3)$.
As the Hodge conjecture is known for $X$ and its powers \cite{ShioC}, this decomposition holds true at the level of homological motives
\[M(X)=\one  \oplus \one(-1) \oplus \one(-2)^{\oplus 21} \oplus V \oplus \one(-3) \oplus \one(-4).\]
This implies that the primitive part of the motive  is of the form
\[\mathfrak{h}^{4,\prim}(X)= \one(-2)^{\oplus 20} \oplus^{\perp} V.\]
Note that the decomposition is orthogonal with respect to the cup product as the types of the Hodge structures   are different.
Finally, as the motive of $X$ is finite dimensional \cite[Lemma 5.2]{BLP}, this decomposition lifts to the level of Chow motives. (Alternatively, see Remark \ref{hypothesismainthm}(\ref{homvsrat}).)

Let us now work over $\overline{\mathbb{F}}_p$. (This is enough for our purpose, thanks to Proposition \ref{reducefinite}).  The positivity of the cup product on algebraic classes on the factor $\one(-2)^{\oplus 20}$ is clear as all these classes come from characteristic zero, see Remark \ref{remHR}. 
We are reduced to the study of algebraic classes on the two dimensional motive $V$.  As the characteristic polynomial of   Frobenius acting on $V$ is a rational polynomial of degree two, there are either zero or two rational solutions.  In the first case  the space of algebraic classes on $V$ is reduced to  zero hence  the standard conjecture of Hodge type holds trivially. In the second case the Fermat variety is supersingular and $V$ is spanned by algebraic cycles  \cite{ShioF}. In this case the standard conjecture of Hodge type holds true via Theorem \ref{mainthm}. (Note that there are infinitely many primes for which the non-trivial case occurs \cite[Theorem 2.10]{ShioF}). 
\end{proof}
\begin{remark}
Let us comment on applications and limits of Theorem \ref{mainthm}.
\begin{enumerate}
\item Theorem \ref{mainthm} cannot be applied to show the standard conjecture of Hodge type for abelian varieties of dimension at least five. Indeed, let $A$ be a simple abelian variety  of dimension $g$ and let $M_I\subset \mathfrak{h}^{2i}(A)$ be a factor as constructed in Proposition \ref{primadecomposizione}. By its very construction, the dimension of  $M_I$ is at least $g/i$ as $\Gal(\overline{\Q}/\Q)$ acts transitively on $\Sigma$, see Notation \ref{NotationCM}. Hence the rank of $M_I$ will never be two (except possibly in  middle degree).
\item It seems likely that using Theorem \ref{mainthm} one can show the standard conjecture of Hodge type for some special varieties as we did in Proposition \ref{propfermat} for  Fermat's  cubic fourfold.  On the other hand we do not know examples (other than abelian fourfolds) where this strategy applies  for a whole family of varieties   and we expect such examples to be rare\footnote{For example  the proof of Proposition \ref{propfermat} cannot apply to all cubic fourfolds as the smallest $\Q$-Hodge structure containing the $H^{1,3}$ of a cubic fourfold has in general dimension greater than two.}.  It is rather a miracle, based on the computations of Section 7, that for all abelian fourfolds only motives of rank two turn out to be significant.\end{enumerate}
\end{remark}
\end{appendices}

\bibliographystyle{alpha}	
\bibliography{masterbib}

\end{document}